\documentclass[11pt]{amsart}
\usepackage{amssymb,amsmath,epsfig,mathrsfs, enumerate, xparse, mathtools}
\usepackage[pagewise]{lineno}
\usepackage[nodisplayskipstretch]{setspace}
\usepackage{graphicx}\usepackage[normalem]{ulem}
\usepackage{fancyhdr}
\usepackage[margin=2.5cm]{geometry}
\usepackage{hyperref}
\hypersetup{
colorlinks   = true,
urlcolor     = blue,
linkcolor    = blue,
citecolor   = red ,
bookmarksopen=true
}

\theoremstyle{plain}
\newtheorem{thrm}{Theorem}[section]
\newtheorem{lemma}[thrm]{Lemma}

\newtheorem{cor}[thrm]{Corollary}
\newtheorem{rmrk}[thrm]{Remark}

\allowdisplaybreaks
\begin{document}
\newcommand{\sn}{\mathbb{S}^{n-1}}
\newcommand{\SL}{\mathcal L^{1,p}( D)}
\newcommand{\Lp}{L^p( Dega)}
\newcommand{\CO}{C^\infty_0( \Omega)}
\newcommand{\Rn}{\mathbb R^n}
\newcommand{\Rm}{\mathbb R^m}
\newcommand{\R}{\mathbb R}
\newcommand{\Om}{\Omega}
\newcommand{\Hn}{\mathbb H^n}
\newcommand{\A}{\alpha }
\newcommand{\B}{\beta}
\newcommand{\eps}{\ve}
\newcommand{\BVX}{BV_X(\Omega)}
\newcommand{\p}{\partial}
\newcommand{\IO}{\int_\Omega}
\newcommand{\bG}{\boldsymbol{G}}
\newcommand{\bg}{\mathfrak g}
\newcommand{\bz}{\mathfrak z}
\newcommand{\bv}{\mathfrak v}
\newcommand{\Bux}{\mbox{Box}}
\newcommand{\e}{\ve}
\newcommand{\X}{\mathcal X}
\newcommand{\Y}{\mathcal Y}
\newcommand{\la}{\lambda}
\newcommand{\vf}{\varphi}
\newcommand{\rhh}{|\nabla_H \rho|}
\newcommand{\Ba}{\mathcal{B}_\beta}
\newcommand{\Za}{Z_\beta}
\newcommand{\ra}{\rho_\beta}
\newcommand{\n}{\nabla}
\newcommand{\vt}{\vartheta}
\newcommand{\its}{\int_{\{y=0\}}}
\newcommand{\py}{\partial_y^a}

\numberwithin{equation}{section}

\newcommand{\RN} {\mathbb{R}^N}
\newcommand{\Sob}{S^{1,p}(\Omega)}
\newcommand{\Dxk}{\frac{\partial}{\partial x_k}}
\newcommand{\Co}{C^\infty_0(\Omega)}
\newcommand{\Je}{J_\ve}
\newcommand{\beq}{\begin{equation}}
\newcommand{\bea}[1]{\begin{array}{#1} }
	\newcommand{\eeq}{ \end{equation}}
\newcommand{\ea}{ \end{array}}
\newcommand{\eh}{\ve h}
\newcommand{\Dxi}{\frac{\partial}{\partial x_{i}}}
\newcommand{\Dyi}{\frac{\partial}{\partial y_{i}}}
\newcommand{\Dt}{\frac{\partial}{\partial t}}
\newcommand{\aBa}{(\alpha+1)B}
\newcommand{\GF}{\psi^{1+\frac{1}{2\alpha}}}
\newcommand{\GS}{\psi^{\frac12}}
\newcommand{\HFF}{\frac{\psi}{\rho}}
\newcommand{\HSS}{\frac{\psi}{\rho}}
\newcommand{\HFS}{\rho\psi^{\frac12-\frac{1}{2\alpha}}}
\newcommand{\HSF}{\frac{\psi^{\frac32+\frac{1}{2\alpha}}}{\rho}}
\newcommand{\AF}{\rho}
\newcommand{\AR}{\rho{\psi}^{\frac{1}{2}+\frac{1}{2\alpha}}}
\newcommand{\PF}{\alpha\frac{\psi}{|x|}}
\newcommand{\PS}{\alpha\frac{\psi}{\rho}}
\newcommand{\ds}{\displaystyle}
\newcommand{\Zt}{{\mathcal Z}^{t}}
\newcommand{\XPSI}{2\alpha\psi \begin{pmatrix} \frac{x}{\left< x \right>^2}\\ 0 \end{pmatrix} - 2\alpha\frac{{\psi}^2}{\rho^2}\begin{pmatrix} x \\ (\alpha +1)|x|^{-\alpha}y \end{pmatrix}}
\newcommand{\Z}{ \begin{pmatrix} x \\ (\alpha + 1)|x|^{-\alpha}y \end{pmatrix} }
\newcommand{\ZZ}{ \begin{pmatrix} xx^{t} & (\alpha + 1)|x|^{-\alpha}x y^{t}\\
	(\alpha + 1)|x|^{-\alpha}x^{t} y &   (\alpha + 1)^2  |x|^{-2\alpha}yy^{t}\end{pmatrix}}
\newcommand{\norm}[1]{\lVert#1 \rVert}
\newcommand{\ve}{\varepsilon}
\newcommand{\D}{\operatorname{div}}
\newcommand{\G}{\mathscr{G}}
\newcommand{\W}{\tilde{W}}

\title[quantitative uniqueness etc.]{  Quantitative uniqueness for fractional heat type operators}

\author{Vedansh Arya}
\address{Tata Institute of Fundamental Research\\
Centre For Applicable Mathematics \\ Bangalore-560065, India}\email[Vedansh Arya]{vedansh@tifrbng.res.in}

\author{Agnid Banerjee}
\address{Tata Institute of Fundamental Research\\
Centre For Applicable Mathematics \\ Bangalore-560065, India}\email[Agnid Banerjee]{agnidban@gmail.com}

\thanks{Second author is supported in part by SERB Matrix grant MTR/2018/000267 and by Department of Atomic Energy,  Government of India, under
project no.  12-R \& D-TFR-5.01-0520.}


%
%
%
\keywords{}
\subjclass{35A02, 35B60, 35K05}

\maketitle
\begin{abstract}
In this paper we obtain quantitative bounds on the maximal order of vanishing for solutions to $(\partial_t - \Delta)^s u =Vu$ for $s\in [1/2, 1)$ via new Carleman estimates. Our main results  Theorem \ref{mainthrm} and Theorem \ref{mainthrm2}  can be thought of as a parabolic generalization of the  corresponding quantitative uniqueness result in the  time independent case  due to R\"uland and it can also be regarded as  a nonlocal generalization of a similar  result due to  Zhu for solutions to local parabolic equations.

\end{abstract}

\tableofcontents

\section{Introduction and the Statement of the main result}

We say that the vanishing order of a function $u$ is $\ell$ at $x_0$, if $\ell$ is the largest  integer such that $D^{\alpha} u=0$ for all $|\alpha| \leq \ell$, where $\alpha$ is a multi-index. In the papers \cite{DF1}, \cite{DF2}, Donnelly and Fefferman showed that if $u$ is an eigenfunction with eigenvalue  $\lambda$ on a  smooth, compact and connected $n$-dimensional Riemannian manifold $M$, then  the maximal vanishing order of $u$ is less than $C \sqrt{\lambda},$ where $C$ only depends on the manifold $M$. Using this estimate, they showed  that $H^{n-1}(\{x: u_{\lambda}(x)=0\})\leq C  \sqrt{\lambda},$ where $u_{\lambda}$ is the eigenfunction corresponding to $\lambda$ and therefore  gave a complete  answer to a famous  conjecture of Yau (\cite{Yau}). We note that the zero set of $u_{\lambda}$ is referred to as the nodal set. This order of vanishing is sharp. If, in fact, we consider $M = \mathbb S^n \subset \R^{n+1}$, and we take the spherical harmonic $Y_\kappa$ given by the restriction to $\mathbb S^n$ of the function $f(x_1,...,x_n,x_{n+1}) = \Re (x_1 + i x_2)^\kappa$, then one has $\Delta_{\mathbb S^n} Y_\kappa = - \lambda_\kappa Y_\kappa$, with $\lambda_\kappa = \kappa(\kappa+n-2)$, and the order of vanishing of $Y_\kappa$ at the North pole $(0,...,0,1)$ is precisely $\kappa = C \sqrt {\lambda_\kappa}$. 

In his work  \cite{Ku}  Kukavica considered the more general problem
\begin{equation}\label{e1}
\Delta u = V(x) u,
\end{equation}
where $V\in W^{1, \infty}$, and showed that the maximal vanishing order  of $u$ is bounded above by $C( 1+ ||V||_{W^{1, \infty}})$. He also conjectured that the rate of vanishing order of $u$ is less than or equal to $C(1+ ||V||_{L^{\infty}}^{1/2})$, which agrees with the Donnelly-Fefferman result when $V = - \lambda$. Employing Carleman estimates,  Kenig in \cite{K} showed that the rate of  vanishing order  of $u$  is less than $C(1+ ||V||_{L^{\infty}}^{2/3})$, and furthermore the exponent $\frac{2}{3}$ is sharp for complex  potentials $V$ based on a counterexample of Meshov (see \cite{Me}).

Not so long ago, the rate of vanishing order of $u$ has been shown to be less than $C(1+ ||V||_{W^{1, \infty}}^{1/2})$  independently by Bakri in \cite{Bk} and Zhu in \cite{Zhu1}. Bakri's approach is based on  an extension of the Carleman method in \cite{DF1}.  On the other hand, Zhu's approach is based on a variant of the frequency function approach  employed by Garofalo and Lin in \cite{GL1}, \cite{GL2}), in the context of strong  unique  continuation problems. The approach of Zhu has been  subsequently extended in \cite{BG2}  to variable coefficient principal part with Lipschitz coefficients  where a similar quantitative  uniqueness result at the boundary of $C^{1, Dini}$ domains has been obtained. We would also like to mention that  in \cite{Zhu2}, an analogous  quantitative uniqueness result  has been established for solutions  to parabolic equations of the type
\[
\D(A(x,t) \nabla u) - u_t= Vu,\]
where $V \in C^{1}$  and $A(x,t) \in C^{2}$ by an adaption of an approach due to Vessella in \cite{Ve} ( see also \cite{EV}).
Now for nonlocal equations of the type 
\[
(-\Delta)^s u = Vu,\]
R\"uland in \cite{Ru} showed that the vanishing order is proportional to $C_1||V||_{C^{1}}^{1/2s} + C_2$  which in the limit as $s \to 1$, exactly reproduces the result of Donnelly and Fefferman. See also \cite{Zhu0} for vanishing order estimates for Steklov eigenfunctions which via the extension approach of Caffarelli and Silvestre in \cite{CS},   is relevant to the case $s=1/2$.  See also \cite{BL} for  earlier results on quantitative uniqueness for Steklov eigenvalue problems. We also refer to \cite{FF, Ru1} for other qualitative strong unique continuation results in the nonlocal elliptic setting.

In this work, we are interested in quantitative uniqueness for the following nonlocal equation
\begin{align}\label{e0}
	(\partial_t-\Delta)^s u =V(x,t)u \hspace{2mm} \text{in} \hspace{2mm} \R^n \times \R, \ \text{$s \in [1/2, 1)$},
\end{align} 
where 
\begin{equation}\label{vasump}
\begin{cases}
\text{$V  \in C_{x,t}^{1}(\R^n \times \R)$ for $s>1/2$}
\\
\text{and $V \in C_{x}^{1, \alpha} (\R^n \times \R) \cap C_{x,t}^{1}(\R^n \times \R)$ for some $\alpha >0$ when $s=1/2$}.\end{cases}\end{equation}   We note that the qualitative $C^{1, \alpha}$ assumption on $V$ that we enforce is only to ensure that the solution to the extended problem \eqref{extpr} is $C^{2}$ in the tangential directions.  
More precisely, we study functions $u \in \operatorname{Dom}(H^s) = \{u\in L^2(\Rn \times \R)\mid H^s u \in L^2(\Rn \times \R)\}$ that  satisfy \eqref{e0} above.
Our main result is the following. 
\begin{thrm}\label{mainthrm}
	Let $u$ be a non-trivial solution of \eqref{e0} with $||u||_{L^2(\R^n\times \R)} \le 1$ and where $V$ satisfies the assumptions in \eqref{vasump}.   Then there exist  constants $C_1= C_1(n,s)$ and $C_2=C_2(n, s, u)$
	such that for all $\rho$ small enough, we have 
	\begin{equation}\label{main1} ||u||_{C^2(B_{\rho} \times (-2,2) )}  \geq C_1 \rho^{A_0},\end{equation}
	where $A_0=C_1||V||_{C^{1}_{x,t}(\Rn \times \R)}^{1/2s}+C_2$. Over here, the $C^{2}$ norm of $u$ is  the parabolic $C^{2}$ norm  defined as
	\begin{align*}
	&|| u ||_{C^2(B_{\rho} \times (-1,0) )} \overset{def}{=}  || u ||_{L^{\infty}(B_{\rho} \times (-1,0) )} + ||\nabla_x u||_{L^{\infty}(B_{\rho} \times (-1,0) )} + || \nabla_x^2 u||_{L^{\infty}(B_{\rho} \times (-1,0) )} +  ||u_t||_{L^{\infty}(B_{\rho} \times (-1,0) )}.\end{align*}
	
	\end{thrm}
\begin{rmrk}
We note that it suffices to assume that \eqref{e0} holds locally for the validity of  Theorem \ref{mainthrm}. We have  assumed that \eqref{e0} holds globally in $\Rn \times \R$ only for the simplicity of the exposition. We also note that the qualitative $C^{1, \alpha}$ assumption on $V$ that we impose when $s=1/2$ is to ensure that the solution to \eqref{e0}  is $C^{2}$ for the estimate in Theorem \ref{mainthrm} to make sense. In this context, we refer to Lemma \ref{reg1} below. \end{rmrk}
In the case when $V \in C^{2}_{x,t}(\Rn \times \R)$, we have the following improvement of Theorem \ref{mainthrm} where the $C^2$ norm of $u$  in the estimate \eqref{main1} can be replaced by the $L^2$ norm.
\begin{thrm}\label{mainthrm2}
Let $u$ be a non-trivial solution of \eqref{e0} with $V \in C^2_{x,t}$ such that $||u||_{L^2(\R^n\times \R)} \le 1.$  Then there exist  constants $\mathcal C_1=  \mathcal C_1
( n, s)$ and $\mathcal C_2 = \mathcal C_2(u,n,s)$ such that for all $\rho$ small enough, we have 
\begin{equation}\label{main2} ||u||_{L^2(B_{\rho} \times (-2,2) )}  \geq \mathcal C_1 \rho^{M},\end{equation}
where $M=\mathcal C_1||V||_{C^{2}_{x,t}(\Rn \times \R)}^{1/2s}+ \mathcal C_2$.
\end{thrm}

Our results in Theorem \ref{mainthrm} and Theorem \ref{mainthrm2}  can thus be viewed as a generalization of the results in \cite{Ru} and \cite{Zhu2}.  To the best of our knowledge, this is the first quantitative uniqueness result for fractional heat type operators. To provide some further perspective, we mention that for global solutions of the nonlocal equation \eqref{e0} a backward space-time strong unique continuation theorem was previously established by one  of us with Garofalo in \cite{BG}. Such result represented the nonlocal counterpart of the one first obtained  by Poon in \cite{Po} for the local case $s=1$. Very recently in \cite{ABDG}, a space like unique continuation property for local solutions  to equations of the type \eqref{e0} has been established by   both of us in a joint work with Danielli and Garofalo which constitutes the nonlocal counterparts of the space like strong unique continuation results in \cite{EF, EFV} for the local case $s=1$.
The following are the key steps in the proof of our main result.

\medskip

\emph{Step 1:} We first establish a quantitative  Carleman estimate  for solutions to a certain parabolic extension problem  ( see \eqref{extpr} below) using which we derive a quantitative vanishing order estimate at the bulk. 

\medskip

\emph{Step 2:} We then establish a propagation of smallness estimate  from the boundary to the bulk by  means of a  new Carleman estimate which is a  subtle generalization of the Carleman estimate in \cite{RS}  to the time dependent case. It turns out that such  an estimate only  holds for $s \geq 1/2$. This is precisely where we require $s\geq1/2$.

\medskip

\emph{Step 3:} Then by combining the bulk quantitative estimate  for a local extension problem  as in \emph{Step 1} with the propagation of smallness estimate in \emph{Step 2}, we establish a vanishing order estimate at the boundary for the extension problem \eqref{extpr} which implies the estimate in Theorem \ref{mainthrm} for the nonlocal problem \eqref{e0}.

\emph{Step 4:} Theorem \ref{mainthrm2} is then obtained from Theorem \ref{mainthrm} by means  of  interpolation type inequalities proven in Lemma \ref{interpolation} which is well suited for our  parabolic situation combined with improved  regularity estimates in Lemma \ref{reg1} that we obtain  when $V \in C^2$. We note that similar interpolation type inequalities have previously appeared  in   \cite{BL, Ru, RS}.

The reader will find that  although the  proof of Theorem \ref{mainthrm}  is inspired by ideas in \cite{Bk, BGM, Ru, Zhu2, Ve}, it has  nevertheless required  some delicate adaptations in our situation because of new novel challenges in the nonlocal parabolic case.

The paper is organized as follows. In section \ref{s:n}, we introduce some basic notations and gather some preliminary results that are relevant to our work.   In section \ref{s:m}, we prove our key results, Lemmas \ref{carlmanbulk}, \ref{bulkdf}, \ref{carlbdr} and \ref{bbinterf} which constitute the novel part of our work from which the main results follow.

In closing, we would like to mention  that the study of the fractional heat type operators as well as the related  extension problem has received a lot of attention in recent times, see  for instance \cite{Au}, \cite{ACM}, \cite{AT}, \cite{BDGP1}, \cite{BDGP2}, \cite{BDS}, \cite{BGMN},  \cite{BSt}, \cite{DP}, \cite{Gcm}, \cite{HSSW}, \cite{LLR} and \cite{LN}.  
\section{Notations and Preliminaries}\label{s:n}
In this section we introduce the relevant notation and gather some auxiliary  results that will be useful in the rest of the paper. Generic points in $\Rn \times \R$ will be denoted by $(x_0, t_0), (x,t)$, etc. For an open set $\Omega\subset \Rn_x\times \R_t$ we indicate with $C_0^{\infty}(\Omega)$ the set of compactly supported smooth functions in $\Om$.  The symbol $\mathscr S(\R^{n+1})$ will denote the Schwartz space of rapidly decreasing functions in $\R^{n+1}$. For $f\in \mathscr S(\R^{n+1})$ we denote its Fourier transform by 
\[
 \hat f(\xi,\sigma) = \int_{\Rn\times \R} e^{-2\pi i(\langle \xi,x\rangle + \sigma t)} f(x,t) dx dt = \mathscr F_{x\to\xi}(\mathscr F_{t\to\sigma} f).
\]
The heat operator in $\R^{n+1} = \Rn_x \times \R_t$ will be denoted by $H = \p_t - \Delta_x$. Given a number $s\in (0,1)$ the notation $H^s$ will indicate the fractional power of $H$ that in  \cite[formula (2.1)]{Sam} was defined on a function $f\in \mathscr S(\R^{n+1})$ by the formula
\begin{equation}\label{sHft}
\widehat{H^s f}(\xi,\sigma) = (4\pi^2 |\xi|^2 + 2\pi i \sigma)^s\  \hat f(\xi,\sigma),
\end{equation}
with the understanding that we have chosen the principal branch of the complex function $z\to z^s$. We then introduce the natural domain for the operator $H^s$.
\begin{align}\label{dom}
\mathscr H^{2s} & =  \operatorname{Dom}(H^s)   = \{f\in \mathscr S'(\R^{n+1})\mid f, H^s f \in L^2(\R^{n+1})\}
\\
&  = \{f\in L^2(\R^{n+1})\mid (\xi,\sigma) \to (4\pi^2 |\xi|^2 + 2\pi i \sigma)^s  \hat f(\xi,\sigma)\in L^2(\R^{n+1})\},
\notag
\end{align} 
where the second equality is justified by \eqref{sHft} and Plancherel theorem. 
 It is important to keep in mind that definition \eqref{sHft} is equivalent to the one based on Balakrishnan formula (see \cite[(9.63) on p. 285]{Samko})
\begin{equation}\label{balah}
H^s f(x,t) = - \frac{s}{\Gamma(1-s)} \int_0^\infty \frac{1}{\tau^{1+s}} \big(P^H_\tau f(x,t) - f(x,t)\big) d\tau,
\end{equation}
where we have denoted by
\begin{equation}\label{evolutivesemi}
P^H_\tau f(x,t) = \int_{\Rn} G(x-y,\tau) f(y,t-\tau) dy = G(\cdot,\tau) \star f(\cdot,t-\tau)(x)
\end{equation}
the \emph{evolutive semigroup}, see \cite[(9.58) on p. 284]{Samko}. We refer to Section 3 in \cite{BG} for relevant details.

Henceforth, given a point $(x,t)\in \R^{n+1}$ we will consider the thick half-space $\R^{n+1} \times \R^+_y$. At times it will be convenient to combine the additional variable $y>0$ with $x\in \Rn$ and denote the generic point in the thick space $\Rn_x\times\R^+_y$ with the letter $X=(x,y)$. For $x_0\in \Rn$ and $r>0$ we let $B_r(x_0) = \{x\in \Rn\mid |x-x_0|<r\}$,
$\mathbb B_r(x_0,0)=\{X = (x,y) \in \R^n \times \R^{+}\mid |x-x_0|^2 + y^2 < r^2\}$ (note that this is the upper half-ball).
When the center $x_0$ of $B_r(x_0)$ is not explicitly indicated, then we are taking $x_0 = 0$. Similar agreement for the thick half-balls $\mathbb B_r(x_0,0)$. We will also use the $\mathbb Q_{r}$ for the set $\mathbb B_r \times (t_0-r^2,t_0+r^2)$ and $Q_r$ for the set  $ B_r \times (t_0-r^2,t_0+r^2).$
For notational ease $\nabla U$ and  $\operatorname{div} U$ will respectively refer to the quantities  $\nabla_X U$ and $ \operatorname{div}_X U$.  The partial derivative in $t$ will be denoted by $\p_t U$ and also at times  by $U_t$. The partial derivative $\partial_{x_i} U$  will be denoted by $U_i$. At times,  the partial derivative $\partial_{y} U$  will be denoted by $U_{n+1}$.

 We next introduce the extension problem associated with $H^s$.  
 Given a number $a\in (-1,1)$ and a $u:\R^n_x\times \R_t\to \R$ we seek a function $U:\R^n_x\times\R_t\times \R_y^+\to \R$ that satisfies the boundary-value problem
\begin{equation}\label{la}
\begin{cases}
\mathscr{L}_a U \overset{def}{=} \partial_t (y^a U) - \operatorname{div} (y^a \nabla U) = 0,
\\
U((x,t),0) = u(x,t),\ \ \ \ \ \ \ \ \ \ \ (x,t)\in \R^{n+1}.
\end{cases}
\end{equation}
The most basic property of the Dirichlet problem \eqref{la} is that if $s = \frac{1-a}2\in (0,1)$ and $u \in \text{Dom}(H^{s})$, then we have the following convergence  in $L^{2}(\R^{n+1})$
\begin{equation}\label{np}
2^{-a}\frac{\Gamma(\frac{1-a}{2})}{\Gamma(\frac{1+a}{2})} \py U((x,t),0)=  - H^s u(x,t),
\end{equation}
where $\py$ denotes the weighted normal derivative
\begin{equation}\label{nder}
\py U((x,t),0)\overset{def}{=}   \operatorname{lim}_{y \to 0^+}  y^a \partial_y U((x,t),y).
\end{equation}

When $a = 0$ ($s = 1/2$) the problem \eqref{la} was first introduced in \cite{Jr1} by Frank Jones, who in such case also constructed the relevant Poisson kernel and proved \eqref{np}. More recently Nystr\"om and Sande in \cite{NS} and Stinga and Torrea in \cite{ST} have independently extended the results in \cite{Jr1} to all $a\in (-1,1)$. 

Therefore, if $u$ is a solution to \eqref{e0},  then it follows that the extended $U$ is a weak solution to   
\begin{equation}\label{wk}
\begin{cases}
\mathscr{L}_a U=0 \ \ \ \ \ \ \ \ \ \ \ \ \ \ \ \ \ \ \ \ \ \ \ \ \ \ \ \ \ \ \text{in}\ \R^{n+1}\times \R^+_y,
\\
U((x,t),0)= u(x,t)\ \ \ \ \ \ \ \ \ \ \ \ \ \ \ \ \text{for}\ (x,t)\in \R^{n+1},
\\
\py U((x,t),0)=  V(x,t) u(x,t)\ \ \ \ \text{for}\ (x,t)\in \R^{n+1}.
\end{cases}
\end{equation}
We note that the $V$ in \eqref{wk} differs from the $V$ in \eqref{e0} by a multiplicative constant.  For notational purposes it will be convenient to work with the following backward version of problem \eqref{wk} \begin{equation}\label{extpr}
\begin{cases}
y^a \partial_t U + y^a \operatorname{div}(y^a \nabla U)=0\  \text{in} \ \{y>0\},
\\	
U((x,t),0)= u(x,t)
\\
\py U((x,t),0)= V(x,t) u(x,t).
\end{cases}
\end{equation}
We note that the former can be transformed into the latter by changing $t \to -t$.

We now record a regularity result which is crucial to our analysis.  Prior to that, we would like to mention  that for relevant notions of  parabolic $C^{k}$ and$C^{k, \alpha}$ spaces, we refer the reader to chapter 4 in \cite{Li}.

\begin{lemma}\label{reg1}
Let $U$  be a weak solution of \eqref{extpr} where $V$ satisfies \eqref{vasump}. Then there exists $\alpha'>0$ such that one has up to the thin set $\{y=0\}$ 
\[
U_i,\ U_t,\ y^a U_y\ \in\ C^{\alpha'}_{loc}.
\]
Moreover, the relevant H\"older norms  over a compact set $K$ are bounded by $\int U^2 y^a dX dt$ over a larger set $K'$ which contains $K$. We also  have that $\nabla_x^2 U \in C^{\alpha'}_{loc}$ up to the thin  set $\{y=0\}$. Furthermore, when $V \in C^{2}_{x,t}(\Rn \times \R)$, we have that the following estimate holds for $i, j=1, ..,n$
\begin{equation}\label{ret}
\int_{\mathbb B_1 \times (-1, 0]} (U_t^2+ U_{tt}^2) y^a +\int_{\mathbb B_1 \times (-1, 0]} |\nabla U_{t}|^2 y^a  +  \int_{\mathbb B_1 \times (-1, 0]} |\nabla U_{ij}|^2 y^a \leq C( 1+ ||V||_{C^2}) \int_{\mathbb B_2 \times (-4, 0]} U^2 y^a.\end{equation}

\end{lemma}  

\begin{proof}
Note that in the case when $s=1/2$ where the extension operator is the heat operator, it follows from the classical theory as in Chapter 6 in \cite{Li}. Therefore we only focus on the case $s>1/2$. In such a case, by  arguing as in the proof of Lemma 5.5 in \cite{BG} using repeated incremental quotients, we deduce that for each $i=1,..,n$,  $U_i=w$ solves
\begin{equation}\label{kojk}
\begin{cases}
\D(y^a \nabla  w) + y^a w_t=0,
\\
\py w= V w + V_i U= f\in L^{\infty}_{loc}.
\end{cases}
\end{equation}
We can then apply the arguments in \cite{AK} ( see also \cite{BSt}) based on compactness methods to assert that $\nabla w \in H^{\alpha'}$ for some $\alpha'>0$ which implies the desired conclusion.  

Now when $V$ is additionally $C^{2}$, then we can take further  incremental quotients  and finally assert that $w=U_{ij}$ solves
\begin{equation}\label{kojk1}
\begin{cases}
\D(y^a \nabla  w) + y^a w_t=0,
\\
\py w= Vw + V_i U_j + V_j U_i + V_{ij} U.
\end{cases}
\end{equation}
Similarly $w=U_t$ solves
\begin{equation}\label{kojk2}
\begin{cases}
\D(y^a \nabla  w) + y^a w_t=0,
\\
\py w= Vw +  V_t U.
\end{cases}
\end{equation}
Also since $V$ is $C^2$ in time, we can likewise assert  that $w=U_{tt}$ solves
\begin{equation}\label{kojk3}
	\begin{cases}
		\D(y^a \nabla  w) + y^a w_t=0,
		\\
		\py w= Vw + V_t U_t + V_{tt} U,
	\end{cases}
\end{equation}
and moreover 
\begin{equation}\label{bt}
||U_{tt}||_{L^{\infty}(\mathbb B_1 \times (-1, 0])} \leq C(1+||V||_{C^2}) ||y^{a/2} U||_{L^2(\mathbb B_2 \times (-4, 0])}
\end{equation}
Thus from the energy estimate as in the proof of Theorem 5.1 in \cite{BG}, the estimates for $\nabla_x U, \nabla_x^2 U, U_t $ in terms of $||y^{a/2} U||_{L^2(\mathbb B_2 \times (-4, 0])}$ and also by using  \eqref{bt}, we find that \eqref{ret} follows.
\end{proof}

We now state and prove an elementary Rellich type identity that is required  in our analysis. This can be regarded as a slight variant of the one in \cite{CS}. 
\begin{lemma}{(\emph{Rellich type Identity})}\label{Rellich}
Let $F$ be a smooth function with $\text{supp}(F) \subset (\overline{\mathbb B_R}\setminus \{0\}) \times (0,1)$. Then for any $k$ we have 
\begin{align}
	&\int |X|^k\D(y^a\n F)\langle \n F, X\rangle  =-k \int  |X|^{k-2}\langle \n F, X \rangle^2y^a 
	+ \frac{(n+a-1+k)}{2}\int |X|^{k} |\n F|^2 y^a\notag\\& - \int_{\{y=0\}}|x|^k \py F  \langle \n_xF, x \rangle \label{Rellich1}\\
	&\text{and}\notag\\
	&\int |X|^k \D(y^a\n F)F_t=-k\int |X|^{k-2} F_t \langle \n F, X \rangle y^a- \int_{\{y=0\}} |x|^k \py F F_t\label{Rellich2}.
\end{align}
\end{lemma}
\begin{proof}  For notational convenience, we let $r=|X|$ throughout the proof. Since  $\text{supp}(F) \subset \overline{\mathbb B_R}\setminus \{0\} \times (0,1),$ we will only have the boundary term at $\{y=0\}.$ Now by integrating by parts, we have
\begin{align*}
	&\;\;\;\;\;\int \D (y^a \n F) r^k \langle \n F, X \rangle\\
	&= - \int  \langle \n F, \n (r^k \langle \n F, X \rangle\rangle y^a -\int_{\{y=0\}} r^k\py F\langle \n_x F, x \rangle \\
	&=- \int \langle \n F, X \rangle \langle\n F, \n r^k\rangle  y^a  - \int r^k \langle \n F, \n \langle \n F, X \rangle \rangle y^a - \int_{\{y=0\}} r^k\py F \langle \n_x F, x \rangle\\
	&=-k \int \langle \n F, X \rangle^2 r^{k-2}y^a  - \int r^{k} |\n F|^2 y^a  - \int r^k \langle\n F,  \nabla^2 F  X\rangle  y^a -\int_{\{y=0\}} r^k\py F\langle \n_x F, x \rangle \\
	&=-k \int \langle \n F, X \rangle^2 r^{k-2}y^a  - \int r^{k} |\n F|^2 y^a  - \frac{1}{2}\int r^k\langle X, \n (|\n F|^2)\rangle y^a  -\int_{\{y=0\}} r^k\py F\langle \n_x F, x\rangle.
\end{align*}
By applying integration by parts to the  term 
\[
\frac{1}{2}\int r^k\langle X, \n (|\n F|^2)\rangle y^a
\]
and by using $<X, \n r^k>= kr^k$,   $\underset{y \rightarrow 0}{\text{lim}}y^{1+a}r^k|\n F|^2=0$, we get
\begin{align*}
&\int \D (y^a \n F) r^k \langle \n F, X \rangle\\	&= -k \int \langle \n F, X \rangle^2 r^{k-2}y^a  - \int r^{k} |\n F|^2 y^a  + \frac{k}{2}\int r^k (|\n F|^2) y^a \\
	&\;\;\;+ \frac{n+a+1}{2}\int r^k |\n F|^2 y^a -\int_{\{y=0\}} r^k\py F\langle \n_x F, x \rangle\\
	&=-k \int \langle \n F, X \rangle^2 r^{k-2}y^a  
	+ \frac{(n+a-1+k)}{2}\int r^{k} |\n F|^2 y^a  -\int_{\{y=0\}} r^k\py F\langle \n_x F, x \rangle.
\end{align*}
This completes the proof of \eqref{Rellich1}.\\
Now for the proof of \eqref{Rellich2}, we again apply integration by parts to get
\begin{align*}
	\int r^k \D(y^a \n F)F_t&=-\int \langle\n (r^k F_t), \n F\rangle y^a -\int_{\{y=0\}} r^k \py  FF_t\\
	&=-k\int r^{k-2} F_t \langle \n F, X \rangle y^a - \int r^k \langle\n F_t, \n F\rangle y^a -\int_{\{y=0\}} r^k \py FF_t\\
	&= -k\int r^{k-2} F_t \langle \n F, X \rangle y^a - \frac{1}{2}\int  \partial_t(r^k|\n F|^2y^a)-\int_{\{y=0\}} r^k \py FF_t.
\end{align*}
Now using  the fundamental theorem of calculus in the $t$-variable, we deduce that  the second integral on the right hand side in the expression above is zero which consequently  finishes the proof of \eqref{Rellich2}. 
\end{proof}
We need the following interpolation inequality where we use some ideas from \cite{RS}. 
\begin{lemma}\label{interpolation}
	Let $s \in (0,1)$ and $f \in C^2_0(\R^n \times \R_+).$ Then there exists a universal constant $C$ such that for any $0<\eta <1$ the following holds
	\begin{align}\label{inte}
		||\nabla_x f||_{L^{2}(\R^n)} \leq C \eta^s\left( || y^{a/2} \nabla \nabla_x f||_{L^2(\R^n \times \R_+)} +|| y^{a/2} \nabla_x f||_{L^2(\R^n \times \R_+)}\right) + C\eta^{-1} ||f||_{L^2(\R^n)}.
	\end{align}
In particular when $n=1$,  we get
\begin{align}\label{inte1}
	|| f_t||_{L^{2}(\R)} \leq C \eta^s\left( || y^{a/2} \partial_y f_t||_{L^2(\R \times \R_+)}  +||y^{a/2}f_{tt}||_{L^2(\R \times \R_+)}  +||y^{a/2}f_{t}||_{L^2(\R \times \R_+)}\right) + C\eta^{-1} ||f||_{L^2(\R)}.
\end{align}
\end{lemma}
\begin{proof}In the proof, we use the notation $\langle \xi \rangle:= \sqrt{1+|\xi|^2}.$ From Plancherel theorem, we have 
	\begin{align}\label{pc1}
		\int_{\R^n} |\n_x f|^2=\int_{\R^n} |\widehat{\n_x f}|^2.
	\end{align}
Now we write $|\widehat{\n_x f}|^2=\left(\langle \xi \rangle^{2s}|\widehat{\n_x f}|^2\right)^{1-s}\left(\langle \xi \rangle^{-2(1-s)}|\widehat{\n_x f}|^2\right)^s$, then \eqref{pc1} becomes 
\begin{align}\label{pc2}
	\int_{\R^n} |\n_x f|^2=\int_{\R^n}\left(\langle \xi \rangle^{2s}|\widehat{\n_x f}|^2\right)^{1-s}\left(\langle \xi \rangle^{-2(1-s)}|\widehat{\n_x f}|^2\right)^s.
\end{align}
 We now apply Young's inequality \begin{equation}\label{young}AB \le \frac{\mu^p A^p}{p} + \frac{B^q}{\mu^{q}q},\end{equation}
with $A=\left(\langle \xi \rangle^{2s}|\widehat{\n_x f}|^2\right)^{1-s},$ $B=\left(\langle \xi \rangle^{-2(1-s)}|\widehat{\n_x f}|^2\right)^s,$ $p=1/(1-s)$ and $q=1/s$ in the right hand side of \eqref{pc2} to get 
\begin{align}\label{pp2}
	\int_{\R^n} |\n_x f|^2 \le (1-s)\mu^{\frac{1}{1-s}}\int_{\R^n}\langle \xi \rangle^{2s}|\widehat{\n_xf}|^2 + s \mu^{-1/s}\int_{\R^n}\langle \xi \rangle^{-2(1-s)}|\widehat{\n_xf}|^2.
\end{align}
Now we estimate the second term in the right hand side of \eqref{pp2}. We first use $|\widehat{\n_xf}|^2=|\xi|^2|\hat{f}|^2 \le \langle \xi \rangle ^2|\hat{f}|^2$ to get 
\begin{align}\label{sec2}
	 s \mu^{-1/s}\int_{\R^n}\langle \xi \rangle^{-2(1-s)}|\widehat{\n_xf}|^2= s \mu^{-1/s}\int_{\R^n}\left( \langle \xi \rangle^{2s}|\hat{f}|^{2s}\right)\left(|\hat{f}|^{2-2s}\right).
\end{align}
Then we again apply the Young's inequality \eqref{young} with $\mu=\epsilon,$ $A=(\langle \xi \rangle|\hat{f}|)^{2s},$ $B=|\hat{f}|^{2-2s},$ $p=1/s,$ and $q=1/(1-s)$  to obtain
\begin{align}\label{es11}
	s \mu^{-1/s}\int_{\R^n}\langle \xi \rangle^{-2(1-s)}|\widehat{\n_x f}|^2 \le s^2\mu^{-1/s}\epsilon^{1/s}\int_{\R^n}\langle \xi \rangle^2|\hat{f}|^2 +s(1-s)\mu^{-1/s}\epsilon^{-\frac{1}{1-s}}\int_{\R^n}|\hat{f}|^2.
\end{align}
Since $|\widehat{\n_x f}| = |\xi||\hat{f}|,$ therefore $\langle \xi \rangle^2|\hat{f}|^2= |\hat{f}|^2+|\widehat{\n_xf}|^2.$
We now choose $\epsilon$ such that 
\begin{align}\label{epc}
	s^2\mu^{-1/s}\epsilon^{1/s}=\frac{1}{2} \implies \epsilon =\frac{\mu}{(2s^2)^s}.
	\end{align}
By substituting  this value of  $\epsilon$ in  \eqref{es11}, we get
\begin{align}\label{es22}
	s \mu^{-1/s}\int_{\R^n}\langle \xi \rangle^{-2(1-s)}|\widehat{\n_xf}|^2 \le \frac{1}{2}\int_{\R^n}|\widehat{\n_xf}|^2 +\frac{1}{2}\int_{\R^n}|\hat{f}|^2 +s(1-s)(2s^2)^{s/(1-s)}\mu^{-1/s}\mu^{-\frac{1}{1-s}}\int_{\R^n}|\hat{f}|^2.
\end{align} 
Then by using \eqref{es22} in \eqref{pp2}  we find
\begin{align}\label{pul}
	\frac{1}{2} \int_{\R^n} |\widehat{\n_xf}|^2 \le (1-s)\mu^{\frac{1}{1-s}}\int_{\R^n}\langle \xi \rangle^{2s}|\widehat{\n_xf}|^2 +\frac{1}{2}\int_{\R^n}|\hat{f}|^2 + (1-s)2^{\frac{s}{1-s}}s^{\frac{1+s}{1-s}}\mu^{\frac{-1}{s(1-s)}}\int_{\R^n}|\hat{f}|^2.
	\end{align}
 Now using the trace inequality  as in Lemma 4.4 in \cite{RS}, we can estimate$\int_{\R^n}\langle \xi \rangle^{2s}|\widehat{\n_xf}|^2$ as follows
 \begin{equation}\label{try1}\int_{\R^n}\langle \xi \rangle^{2s}|\widehat{\n_xf}|^2 \le C\int_{\R^n \times \R_+}|\n\n_x f|^2y^a+C\int_{\R^n \times \R_+}|\n_x f|^2y^a.\end{equation} The conclusion follows  by employing  the estimate \eqref{try1} in \eqref{pul} and  subsequently  by letting $\mu^{1/(1-s)}=\eta^s.$ 
	\end{proof}
\section{Proof of the main result}\label{s:m}
\subsection{Carleman Estimate I}
We first state and prove our main Carleman estimate using which we prove an upper bound on the vanishing order in the bulk.  This is a generalization of the estimate  in \cite[Theorem 2]{EV}. However, the Carleman weight  that we use is similar to that in \cite[Theorem 2.1]{Bk}. The proof of this estimate is also partly inspired by that of Theorem 1.1 in \cite{BGM}.  
\begin{lemma}\label{carlmanbulk}
Let $\W$ be compactly supported in  $(\overline{\mathbb B_R} \setminus \{0\} )\times (0,1))$ and assume that it solves  
\begin{equation}\label{kref}
\begin{cases}
\D(y^a \n \W) + y^a \W_t =g  \;\;\text{in $\R^{n+1}\times \R_y^+$}
\\
  \py \W= V\W\ \text{on $\{y=0\}$},
  \end{cases}
  \end{equation}
   where $V \in C^{1}(\R^n \times \R)$.
Then there exist universal constants $C=C(n,s)$ and $R_0$ such that for all $r <R_0$ and any  $\A >C(n,s)(1+||V||_{C^{1}}^{1/2s})$, the following estimate holds
\begin{align}\label{carleman}
	\alpha^3\int  |X|^{-2\A -4 + \epsilon} e^{2\A |X|^\epsilon} \W^2 y^a \le C \int |X|^{-2 \A}e^{2\A |X|^\epsilon }  g^2y^{-a}.
\end{align}
\end{lemma}
\begin{proof}Again for notational convenience we let $r=|X|$ and also  denote $||V||_{C^1_{x,t}}$ by $||V||_1$.  We then  set $W=r^{-\B}e^{\A r^\epsilon }\W$ where $\epsilon=(1-a)/4$ and  $\B$ will be chosen depending on $\A$ later. Since $\W=r^{\B}e^{-\A r^\epsilon }W$, from direct calculations we have
\begin{align*}  
	g=\D(y^a \n \W) + y^a \W_t=&(- \A \epsilon (2\B+\epsilon + n+a-1) r^{\B + \epsilon -2}+ \B (\B+n+a-1)r^{\B-2} + \A ^2 \epsilon^2 r^{\B +2 \epsilon -2})e^{-\A r^\epsilon }W y^a\\
	& + 2 ( \B r^{\B-2} - \A \epsilon r^{\B + \epsilon -2}) e^{-\A r^\epsilon } \langle\n W, X\rangle y^a +r^{\B}e^{-\A r^\epsilon } (\D(y^a \n W) + y^a W_t).
\end{align*}
Now using the inequality $(A+B)^2 \geq A^2 + 2 AB $ with
\begin{align*}
	A&=2 ( \B r^{\B-2} - \A \epsilon r^{\B + \epsilon -2}) e^{-\A r^\epsilon } \langle\n W, X\rangle y^a + r^{\B}e^{-\A r^{\epsilon}}W_ty^a
\end{align*}
and with $B$ being the rest of the terms in the expression for $g$ above, we get
\begin{align}\label{ineq}
	&\int r^{-2 \A}e^{2\A r^\epsilon }g^2 y^{-a}\ge \int r^{-2 \A} (2(\B r^{\B-2} - \A \epsilon r^{\B + \epsilon -2}) \langle \n W,X\rangle +r^{\B}W_t)^2y^a \\
	&-4 \A \epsilon (2\B+\epsilon + n+a-1) \int  r^{2\B-2 \alpha + \epsilon -2}(\B r^{-2}- \alpha \epsilon r^{\epsilon-2})W\langle \n W,X\rangle y^a\notag\\
	&+4\B (\B+n+a-1)\int r^{2\B-2\A-2}(\B r^{-2}- \alpha \epsilon r^{\epsilon-2})W\langle \n W, X \rangle  y^a\notag\\
	& +4\A ^2 \epsilon^2 \int  r^{2\B- 2\A +2 \epsilon -2}(\B r^{-2}- \alpha \epsilon r^{\epsilon-2})W \langle \n W,X \rangle y^a\notag\\
	&+4\int r^{2\B-2\A}(\B r^{-2}- \alpha \epsilon r^{\epsilon-2})\langle \n W, X \rangle  \D(y^a \n W)\notag\\
	&-2 \A \epsilon (2\B+n +\epsilon +a-1)  \int  r^{2\B-2 \alpha + \epsilon -2}WW_t y^a\notag\\
	&+2\B (\B+n+a-1)\int r^{2\B-2\A-2}WW_t y^a\notag\\
	& +2\A ^2 \epsilon^2 \int  r^{2\B- 2\A +2 \epsilon -2}WW_t y^a
	+2\int r^{2\B-2\A}  \D(y^a \n W)W_t\notag\\
	&=I_1+I_2+I_3+I_4+I_5+I_6+I_7+I_8+I_9.\notag
\end{align}
We now simplify each of the  integrals separately. To handle $I_2,I_3$ and $I_4,$ we  observe by  an application of  the divergence theorem that the following holds
\begin{align}\label{vrvterm}
	\int  r^{2\B-2\A+\epsilon-4}  \langle \n W, X\rangle  W y^a 
	&= \frac{1}{2}\int  r^{2\B-2\A+\epsilon-4} \langle \n W^2, X \rangle y^a \\
	&= -\frac{1}{2}\int \D (y^a r^{2\B-2\A+\epsilon-4} X) W^2 dX \notag\\
	&= -\frac{(2\B-2\A+\epsilon-4+n+a+1)}{2} \int r^{2\B-2\A+\epsilon-2}   W^2 y^a\notag.
\end{align}
 Also,  by analogous computation as in \eqref{vrvterm} above, we note that in order to equate the following  integral  to zero in \eqref{ineq} above, i.e. \[4\B^2(\B+n+a+1)\int r^{2\B-2\A-4}W\langle \n W, X\rangle,\] we require  $\B$ be related $\A$ in the following way
\begin{align}\label{beta}
	2\B-2\A-4+n+a+1=0.
\end{align}
Moreover by  similar calculations as in \eqref{vrvterm} and by substituting the value of $\B$  in \eqref{beta},  we deduce that there exists  a constant  $C=C(n,a)$ such that for all $\A >C$, we  have that $I_2+I_3+I_4$ can be lower bounded as follows
\begin{align*}
I_2+I_3+I_4\geq	6\A^3\epsilon^2\int   r^{-n-a-1 + \epsilon} W^2 y^a 
	-C\alpha^3\int r^{-n-a-1 +2 \epsilon}W^2 y^a.
\end{align*}
Now if we choose $R_0$ such that $CR_0^{\epsilon} <\epsilon^2$, then  we get 
\begin{equation}\label{iut}I_2+I_3+I_4\geq 5\alpha^3\epsilon^2\int   r^{-n-a-1+ \epsilon} W^2 y^a.\end{equation}

Now we note that from \eqref{kref} it follows that  $\py W=VW$. Using this along with \eqref{Rellich1},  we obtain
\begin{align}\label{i50}
	I_5&=4\int r^{2\B-2\A}(\B r^{-2}- \alpha \epsilon r^{\epsilon-2})\langle \n W, X \rangle  \D(y^a \n W)=4\int r^{-n-a+3}(\B r^{-2}- \alpha \epsilon r^{\epsilon-2})\langle \n W, X \rangle  \D(y^a \n W)\\
	&=4\B(n+a-1) \int r^{-n-a-1}   \langle \n W,X\rangle^2 y^a   -4\A\epsilon(n+a-1-\epsilon) \int r^{-n-a-1+\epsilon}\langle \n W,X\rangle^2 y^a\notag \\
	&\;\;\; -2\A\epsilon^2 \int r^{-n-a+1+\epsilon} |\n W|^2 y^a +4\B \int_{\{y=0\}} r^{-n-a+1}   VW \langle \n_x W,x\rangle - 4\A\epsilon \int_{\{y=0\}}r^{-n-a+1+\epsilon}VW  \langle \n_x W,x\rangle.\notag
\end{align}
We now estimate   the integral
$$\int r^{-n-a+1+\epsilon}|\n W|^2 y^a dX.$$ 
Recall $W= r^{-\beta}e^{\alpha r^{\epsilon}}\W,$
therefore we have $\n W = r^{-\beta}e^{\alpha r^{\epsilon}}\n \W - \beta r^{-\beta-2}e^{\alpha r^{\epsilon}}\W X + \alpha \epsilon r^{-\beta +\epsilon -2} e^{\alpha r^{\epsilon}}\W X.$
Thus 
\begin{align}\label{1}
	&|\n W|^2 = \langle \n W, \n W \rangle\\
	&= \langle r^{-\beta}e^{\alpha r^{\epsilon}}\n \W - \beta r^{-\beta-2}e^{\alpha r^{\epsilon}}\W X + \alpha \epsilon r^{-\beta +\epsilon -2} \W X,r^{-\beta}e^{\alpha r^{\epsilon}}\n \W - \beta r^{-\beta-2}e^{\alpha r^{\epsilon}}\W X + \alpha \epsilon r^{-\beta +\epsilon -2} \W X \rangle \notag\\
	&=  r^{-2\beta}e^{2\alpha r^{\epsilon}}(|\n \W|^2 - 2 \beta r^{-2} \langle \n \W, X \rangle \W + 2 \alpha \epsilon r^{\epsilon-2} \langle \n \W, X \rangle \W +  \beta^2 r^{-2}\W^2 -2 \alpha \beta \epsilon r^{\epsilon -2} \W^2 + \alpha^2 \epsilon^2 r^{2\epsilon -2}\W^2 ).\notag
\end{align}
Now we estimate this term by term. Note that from \eqref{beta}, $-2\B -n-a+1+\epsilon=-2\A-2+\epsilon$ and using divergence theorem, we get
\begin{align}\label{ty0}
	&\int r^{-2\A-2+\epsilon}e^{2\alpha r^{\epsilon}}|\n \W|^2 y^a  \\
	&=-\int   \D(r^{-2\A-2+\epsilon}y^a \n \W)\W - \int_{\{y=0\}}e^{2 \alpha r^{\epsilon}}r^{-2\A-2+\epsilon}W \W^2\notag\\
	&= -\int  r^{-2\A-2+\epsilon}e^{2\alpha r^{\epsilon}}  \D(y^a \n \W)\W - \int \langle\n \W, \n (r^{-2\A-2+\epsilon}e^{2\alpha r^{\epsilon}})\rangle \W y^a\label{urterm}\notag\\
	&\;\;\;-\int_{\{y=0\}}e^{2 \alpha r^{\epsilon}}r^{-2\A-2+\epsilon}V \W^2\notag.
\end{align}
Also  from the  fundamental theorem of calculus in the $t$-variable we have
\begin{equation}\label{ty1}
\int r^{-2\A-2+\epsilon}e^{2\alpha r^{\epsilon}} \W \W_ty^a = \int r^{-2\A-2 +\epsilon}e^{2\alpha r^{\epsilon}} (\W^2)_ty^a=0.\end{equation}
From \eqref{ty0} and \eqref{ty1} it thus follows
\begin{align*}
	\int r^{-2\A-2+\epsilon}e^{2\alpha r^{\epsilon}}|\n \W|^2 y^a &= -\int r^{-2\A-2+\epsilon}e^{2\alpha r^{\epsilon}}   (\D(y^a \n \W)+y^a\W_t)\W   - \int \langle \n \W, \n (r^{-2\A-2+\epsilon}e^{2\alpha r^{\epsilon}})\rangle \W y^a\\
	&\;\;\;-\int_{\{y=0\}}e^{2 \alpha r^{\epsilon}}r^{-2\A-2+\epsilon}V \W^2.
\end{align*}
Now by applying Cauchy-Schwarz inequality, writing $\W$ in terms of $W$  and also by using \eqref{kref} and  \eqref{beta} we get
\begin{align*}
	\int (\D(y^a \n \W)+y^a\W_t)\W r^{-2\A-2+\epsilon}e^{2\alpha r^{\epsilon}} 
	&\le \frac{1}{\alpha} \int   e^{2\alpha r^{\epsilon}} r^{-2 \alpha} (\D(y^a \n \W)+y^a\W_t)^2  y^{-a}+ \alpha \int r^{-2 \alpha -4 +2\epsilon} e^{2\alpha r^{\epsilon}}   \W^2 y^a\\
	&= \frac{1}{\alpha} \int e^{2\alpha r^{\epsilon}} r^{-2 \alpha} g^2 y^{-a}+ \alpha   \int  r^{-n-a-1+2\epsilon} W^2  y^a.
\end{align*}
Now we estimate  the  second term in \eqref{ty0}. We have
\begin{align*}
	\int \langle\n \W, \n (r^{-2\A-2+\epsilon}e^{2\alpha r^{\epsilon}})\rangle \W y^a &=\frac{1}{2}\int \langle\n \W^2, \n (r^{-2\A -2 +\epsilon}e^{2\alpha r^{\epsilon}})\rangle  y^a\\
	&=-\frac{1}{2}\int \W^2\D(\n (r^{-2\A-2+\epsilon}e^{2\A r^\epsilon})y^a)-\frac{1}{2}\int_{\{y=0\}}\W^2(r^{-2\A-2+\epsilon}e^{2\A r^{\epsilon}})_yy^a\\
	&=-\frac{1}{2}\int \W^2\D(\n (r^{-2\A-2+\epsilon}e^{2\A r^\epsilon})y^a).
	\end{align*}
	Over here, we used that 
	\[
	\int_{\{y=0\}}\W^2(r^{-2\A-2+\epsilon}e^{2\A r^{\epsilon}})_yy^a=0.\]	
 Now from direct calculations we have 
	$$ \D(\n (r^{-2\A-2+\epsilon}e^{2\A r^\epsilon}) y^a)=((2\A+2-\epsilon)(2\A-\epsilon-n-a+3)+2\A\epsilon(-4\A -5 +3\epsilon +n+a)r^{\epsilon}+4\A^2\epsilon^2r^{2\epsilon})r^{-2\A-4+\epsilon}e^{2\A r^\epsilon} y^a.$$
	Hence by substituting $	\W=r^{\B}e^{-\A r^\epsilon }W$ and also using \eqref{beta}, we get
	\begin{align}\label{ty5}
			&\int \langle\n \W, \n (r^{-2\A-2+\epsilon}e^{2\alpha r^{\epsilon}})\rangle \W y^a\\ &=-\frac{1}{2} (2\A+2-\epsilon)(2\A-\epsilon-n-a+3)\int r^{-n-a-1+\epsilon} W^2y^a\notag\\
			&\;\;-\A\epsilon(-4\A -5 +3\epsilon +n+a)\int r^{-n-a-1+2\epsilon}W^2 y^a -2\A^2\epsilon^2 \int  r^{-n-a-1+3\epsilon} W^2 y^a\notag\\
			&\ge (-2\A^2 +C\A +CR_0^{\epsilon}\A^2 )\int r^{-n-a-1+\epsilon} W^2y^a\notag\\
			&\ge -\frac{5}{2}\A^2 \int r^{-n-a-1+\epsilon} W^2y^a,\notag
	\end{align}
	provided $R_0$ is small enough and $\alpha$ is sufficiently large.
Hence from \eqref{ty0}-\eqref{ty5}, it follows that  for all $\A >C_0$  where $C_0$ is sufficiently large, we have
 \begin{align}\label{2}
	\int r^{-2\A-2+\epsilon} e^{2\alpha r^\epsilon} |\n \W|^2 y^a 	&\le \frac{1}{\alpha} \int   e^{2\alpha r^{\epsilon}} r^{-2 \alpha} g^2 y^{-a}  + \frac{13}{5}\A^2 \int r^{-n-a-1+\epsilon} W^2y^a-\int_{\{y=0\}} r^{-n-a-1+\epsilon}V W^2.
\end{align}
	Now we consider the integrals corresponding to the second and the third term in \eqref{1}. By integrating by parts, we obtain
	\begin{align}\label{3}
		&-2\B \int r^{-2\A-4+\epsilon}e^{2\alpha r^\epsilon}\langle \n \W, X\rangle \W y^a +2\A\epsilon \int r^{-2 \alpha -4 +\epsilon}e^{2\alpha r^\epsilon}\langle \n \W, X\rangle \W y^a  \\
		&=\B (-2\A-3+\epsilon +n+a)\int r^{-2\A-4+\epsilon}e ^{2\alpha r^\epsilon}\W^2 y^a -\A \epsilon(-2\A-5+2\epsilon +n+a) \int r^{-2\A-4+2\epsilon}e^{2\alpha r^\epsilon}\W^2 y^a \notag\\
		&\le -\frac{3}{2} \A^2 \int r^{-2\A-4+\epsilon}W^2 y^a\ \text{(provided $R_0$ is small enough)}\notag.
	\end{align}
Now concerning the remaining terms in \eqref{1}, the corresponding integrals can be estimated as follows
 \begin{align}\label{4}
	&\int (\beta^2 r^{-2}\W^2 -2 \alpha \beta \epsilon r^{\epsilon -2} \W^2 + \alpha^2 \epsilon^2 r^{2\epsilon -2}\W^2)r^{-2\B}e^{2\A r^{\epsilon}}y^a\notag\\
	&=\B^2\int  r^{-n-a-1+\epsilon}W^2 y^a-2 \alpha \beta \epsilon \int r^{-n-a-1+2\epsilon} W^2y^a + \alpha^2 \epsilon^2 \int  r^{-n-a-1+3\epsilon}W^2y^a \notag\\
	&\le \frac{11}{10} \A^2 \int  r^{-n-a-1+\epsilon}W^2 y^a.
\end{align}
	Hence from \eqref{1}, \eqref{2}, \eqref{3} and \eqref{4}, we get 
	\begin{align}\label{interpo}
		\int r^{-n-a+1+\epsilon} |\n W|^2 y^a &\le \frac{1}{\alpha} \int  e^{2\alpha r^{\epsilon}} r^{-2 \alpha} g^2 y^{-a}  +\frac{11}{5}\A^2 \int  r^{-n-a-1+\epsilon}W^2 y^a-\int_{\{y=0\}} r^{-n-a+1+\epsilon} VW^2.
	\end{align}
	Using \eqref{interpo} in \eqref{i50} we get 
	\begin{align}\label{iut1}
		I_5 \ge &\;\;4\B(n+a-1) \int r^{-n-a-1}   \langle \n W,X\rangle^2 y^a   -4\A\epsilon(n+a-1+\epsilon) \int r^{-n-a-1-\epsilon}\langle \n W,X\rangle^2 y^a \\
		&\;\;\;-2\epsilon^2 \int  e^{2\alpha r^{\epsilon}} r^{-2 \alpha} g^2 y^{-a} -\frac{22}{5}\A^3 \epsilon^2 \int  r^{-n-a-1+\epsilon}W^2 y^a-2\A\epsilon^2\int_{\{y=0\}} r^{-n-a+1+\epsilon} VW^2\notag\\
		&+4\B \int_{\{y=0\}} r^{-n-a+1} VW \langle \n_x W,x\rangle - 4\A\epsilon \int_{\{y=0\}} r^{-n-a+1+\epsilon} VW \langle \n_x W,x\rangle.\notag
	\end{align}
From \eqref{iut} and \eqref{iut1} we obtain
	\begin{align}\label{iut5}
	 &I_2+I_3+I_4+I_5\geq 4\B(n+a-1) \int  r^{-n-a-1} \langle \n W,X\rangle^2  y^a   -4\A\epsilon(n+a-1+\epsilon) \int   r^{-n-a-1-\epsilon}  \langle \n W,X\rangle^2 y^a \\
	&\;\;\;-2\epsilon^2 \int   e^{2\alpha r^{\epsilon}} r^{-2 \alpha} g^2 y^{-a}  +\frac{3}{5} \A^3 \epsilon^2 \int  r^{-n-a-1+\epsilon}W^2 y^a-2\A\epsilon^2\int_{\{y=0\}} r^{-n-a+1+\epsilon} VW^2\notag\\
	&+4\B \int_{\{y=0\}} r^{-n-a+1} VW \langle \n_x W,x\rangle - 4\A\epsilon \int_{\{y=0\}} r^{-n-a+1+\epsilon} VW\langle \n_x W,x\rangle.\notag
\end{align}
	Now the boundary   integral which appears  in \eqref{iut5} can be rewritten in the following way
	\begin{align}\label{bdr1}
		&\int_{\{y=0\}}r^{-n-a+1} VW\langle \n_x W,x\rangle  =\frac{1}{2}\int_{\{y=0\}}r^{-n-a+1} V \langle \n_x W^2,x\rangle \\
		&=-\frac{1}{2}\int_{\{y=0\}}\D_x(V x r^{-n-a+1}) W^2\notag\\
		&=-\frac{1}{2} \int_{\{y=0\}} r^{-n-a+1} W^2 \langle\n_x V, x\rangle -\frac{1-a}{2} \int_{\{y=0\}} r^{-n-a+1}  V W^2.\notag
	\end{align}

	Hence  using \eqref{bdr1} in \eqref{iut5}, we get that for $\A >C$ (if necessary we can increase $C$) and $R_0$ small enough, the following inequality holds
\begin{align}\label{2345}
	I_2+I_3+I_4+I_5\ge &\;4\B(n+a-1) \int  r^{-n-a-1}  \langle \n W,X\rangle^2 y^a   -4\A\epsilon(n+a-1+\epsilon) \int   r^{-n-a-1-\epsilon}  \langle \n W,X\rangle^2  y^a \\
	&\;\;\;-2\epsilon^2 \int   e^{2\alpha r^{\epsilon}} r^{-2 \alpha} g^2 y^{-a} +\frac{3}{5}\A^3 \epsilon^2 \int  r^{-n-a-1+\epsilon}W^2 y^a\notag\\
	& \;\;\;-5\A||V||_1 \int_{\{y=0\}} r^{-n-a+1} W^2.\notag
	\end{align}
Now  by integrating by parts in the $t$-variable, we note that   $$\int r^{2\B-2\A+\epsilon-2}WW_ty^adtdX=\frac{1}{2}\int \partial_t(r^{2\B-2\A+\epsilon-2}W^2y^a)dtdX=0.$$
Thus we have
\begin{align}\label{678}
	I_6=I_7=I_8=0.
\end{align}
We  now estimate $I_9$. Using \eqref{Rellich2}, we have
	\begin{align*}
		I_9=	2	\int r^{2\B-2\A} \D(y^a \n W)W_t&=(4\A-4\B)\int r^{{2\B-2\A}-2} W_t \langle \n W, X \rangle- 2\int_{\{y=0\}}r^{2\B-2\A}VWW_t\\
		&=(4\A-4\B)\int r^{{2\B-2\A}-2} W_t \langle \n W, X \rangle+ \int_{\{y=0\}}r^{2\B-2\A}V_tW^2\\
		& \ge (4\A-4\B)\int r^{{2\B-2\A}-2} W_t \langle \n W, X \rangle -||V||_1 \int_{\{y=0\}} r^{2\B-2\A} W^2,
	\end{align*}
	where in the second line, we integrated by parts in the $t$-variable.\\
	The first integral on RHS $(4\A-4\B)\int r^{{2\B-2\A}-2} W_t \langle \n W, X \rangle$ will be absorb in $I_1$. For that, we use the following algebraic identity 
	$$(A+B+C)^2+2\gamma AC=((1+\gamma)A +B +C)^2-\gamma^2 A^2 -2\gamma A^2 -2\gamma AB, $$
	with $A= 2 \beta r^{\beta -2} \langle \n W, X \rangle$, $B=	-2 \alpha \epsilon r^{\beta + \epsilon -2} \langle \n W, X \rangle$, $C=r^{\beta} W_t$ and $\gamma=(\frac{\alpha}{\beta}-1)$.
	
	Thus we have 
	\begin{align}\label{jk1}
		&\int r^{-2 \alpha}(2 \beta r^{\beta -2} \langle \n W, X \rangle -2 \alpha \epsilon r^{\beta + \epsilon -2} \langle \n W, X \rangle+ r^{\beta} W_t)^2 + (4 \alpha - 4 \beta )\int r^{-2 \alpha + 2 \beta-2} W_t \langle \n W, X \rangle\\
		&=\int r^{-2 \alpha}(2 \alpha r^{\beta -2} \langle \n W, X \rangle -2 \alpha \epsilon r^{\beta + \epsilon -2} \langle \n W, X \rangle+ r^{\beta} W_t)^2  + 4 (\beta^2-\alpha^2) \int r^{-n-a-1}\langle \n W, X \rangle^2\notag\\
		& \;\;\; \; +(8 \alpha^2 \epsilon -8 \alpha \beta \epsilon) \int r^{-n-a-1+ \epsilon}\langle \n W, X \rangle^2.\notag
	\end{align}
	Now from \eqref{beta} it follows that  $8 \alpha^2 \epsilon -8 \alpha \beta \epsilon  \geq 0$. Using this in \eqref{jk1} it follows
 \begin{align}\label{19}
I_1 +I_9 \ge &  4 (\beta^2-\alpha^2) \int r^{-n-a-1}\langle \n W, X \rangle^2	-||V||_1 \int_{\{y=0\}} r^{-n-a+3} W^2.
\end{align}
Note that from \eqref{beta}, $4(\B^2-\A^2) <0,$ therefore the term  $4 (\beta^2-\alpha^2) \int r^{-n-a-1}\langle \n W, X \rangle^2$ in \eqref{19} above is unfavourable. However at this point  we make the crucial and subtle observation that such a  term can be  absorbed in   the  term $4\B(n+a-1) \int r^{-n-a-1}   \langle \n W,X\rangle^2 y^a$  in \eqref{2345} because
$$4(\B^2-\A^2) + 4\B(n+a-1)=8\B-(n+a-3)^2>6\B,$$ for all $\B$ large.
Now by taking $R_0$ small enough, we can ensure that the following  term in \eqref{2345}, i.e.  $4\A\epsilon(n+a-1+\epsilon) \int   r^{-n-a-1-\epsilon}  \langle \n W,X\rangle^2 y^a$ can be estimated in the following way \begin{equation}\label{iut10}
 4\A\epsilon(n+a-1+\epsilon) \int   r^{-n-a-1-\epsilon}  \langle \n W,X\rangle^2 y^a < 3 \beta \int r^{-n-a-1}\langle \n W,X\rangle^2 y^a.
 \end{equation} Hence from \eqref{2345}, \eqref{678}, \eqref{19} and \eqref{iut10} we have for all $\alpha$ large
\begin{align}\label{bulk}
&\int   e^{2\alpha r^{\epsilon}} r^{-2 \alpha} g^2 y^{-a}\ge 	-2\epsilon^2 \int   e^{2\alpha r^{\epsilon}} r^{-2 \alpha}g^2 y^{-a} +\frac{3}{5}\A^3 \epsilon^2 \int  r^{-n-a-1+\epsilon}W^2 y^a\\
	 &\;\;\;-6\A||V||_1 \int_{\{y=0\}} r^{-n-a+1} W^2.\notag
\end{align}

	Now we estimate  the boundary integral $ \int_{\{y=0\}} r^{-n-a+1} W^2$. Using polar coordinates, we have
	\begin{align*}
		\int_{\{y=0\}}  r^{-n-a+1} W^2 &= \int \int \int_{S^{n-1}} r^{-a}W^2(r \omega' ,0)d \omega dr dt.
	\end{align*}
	Now by applying the surface  trace inequality (see Lemma 3.1 in \cite{Ru}) and using $|\n W|^2= W_r^2 + \frac{1}{r^2}|\n_{S^n}W|^2$ which implies $|\n_{S^n}W|^2 \leq r^2|\n W|^2$,  we get that  there exists $C_T=C_T(n,a)$ such that for all $\tau >1$
	\begin{align*}
		\int  r^{-n-a+1} W^2 &\leq C_T \tau^{2-2s} \int \int r^{-a} \int_{S_+^n} \omega_{n+1}^a W^2(r \omega) + C_T \tau^{-2s} \int\int r^{-a} \int_{S_+^n} \omega_{n+1}^a |\n_{S^n}W|^2\\
		&\leq  C_T \tau^{2-2s} \int r^{-a} \int_{S_+^n} \omega_{n+1}^a W^2(r \omega) + C_T \tau^{-2s} \int r^{-a} \int_{S_+^n} \omega_{n+1}^a r^2 |\n W(r \omega)|^2. 
	\end{align*}
	Now by transforming it back to Eucledian coordinates, we get
	\begin{align}\label{traceec}
		\int_{\{y=0\}}  r^{-n-a+1} W^2 &\leq C_T \tau^{2-2s} \int r^{-2a-n}W^2 y^a  + C_T \tau^{-2s}  \int r^{-2a-n+2}|\n W|^2 y^a \\
		&\leq C_T \tau^{2-2s}R_0^{1-a-\epsilon} \int r^{-n-a-1+\epsilon}W^2 y^a + C_T \tau^{-2s} R_0^{1-a-\epsilon} \int r^{-n-a+1+\epsilon}|\n W|^2 y^a.\notag
	\end{align}
In \eqref{traceec}, we used that since $\epsilon = (1-a)/4$ therefore $1-a-\epsilon >0$ which in particular implies that $r^{1-a-\epsilon} \leq R_0^{1-a-\epsilon}.$ Now  using the estimate \eqref{interpo} in \eqref{traceec} we obtain
	\begin{align*}
		\int_{\{y=0\}} r^{-n-a+1} W^2
		& \leq C_T \tau^{2-2s}R_0^{1-a-\epsilon} \int r^{-n-a-1+\epsilon}W^2 y^a\\
		&\;\;\; + C_T \tau^{-2s}R_0^{1-a-\epsilon}\left(\frac{1}{\alpha} \int   e^{2\alpha r^{\epsilon}} r^{-2 \alpha} g^2 y^{-a} + \frac{11}{5} \alpha^2 \int  r^{ -n-a-1+\epsilon}W^2 y^{a} \right)\\
		&\;\;\; +C_T\tau^{-2s}R_0^{1-a-\epsilon}  ||V||_1\int_{\{y=0\}}r^{-n-a+1+\epsilon} W^2.
	\end{align*} 
We now take $\tau$  such that $\tau^{-2s}||V||_1<1$. More precisely, we  let $\tau =||V||_{1}^{1/2s} +1.$  Then by taking  $R_0$ small enough, we can ensure that the term   $C_T\tau^{-2s}R_0^{1-a-\epsilon}  ||V||_1\int_{\{y=0\}}r^{-n-a+1+\epsilon} W^2$ can be absorbed in the left hand side of the above inequality. Consequently we have for new $C_T$ that the following inequality holds
	\begin{align}\label{bdr}
		\int_{\{y=0\}} r^{-n-a+1} W^2
		& \leq C_T \tau^{2-2s}R_0^{1-a-\epsilon} \int r^{-n-a-1+\epsilon}W^2 y^a\notag\\
		&\;\;\; +C_T \tau^{-2s}R_0^{1-a-\epsilon}\left(\frac{1}{\alpha} \int  e^{2\alpha r^{\epsilon}} r^{-2 \alpha} g^2 y^{-a} + \frac{11}{5} \alpha^2 \int r^{ -n-a-1+\epsilon} W^2y^a\right). 
	\end{align} 
Using \eqref{bdr} in \eqref{bulk} and also that $\tau= ||V||_1^{1/2s} +1$, we get 
\begin{align}\label{iu20}
	&	2\int  (\D(y^a \n \W)+\W_t)^2 e^{2\alpha r^{\epsilon}} r^{-2 \alpha}\\&\ge \frac{3}{5}\A^3 \epsilon^2 \int  r^{-n-a-1+\epsilon}W^2 y^a
	-6\A  C_T \tau^{2}R_0^{1-a-\epsilon} \int r^{-n-a-1+\epsilon}W^2y^a\notag \\
	&-6\A C_T R_0^{1-a-\epsilon}\left(\frac{1}{\alpha} \int  (e^{2\alpha r^{\epsilon}} r^{-2 \alpha} g^2 y^{-a} + 2.2 \alpha^2 \int W^2 r^{ -n-a-1+\epsilon}\right).\notag
\end{align}
Now by letting $\alpha= \tau +C$ for a sufficiently large $C$  and  finally by rewriting $W$ in  terms of $\W$ we deduce from \eqref{iu20} that for sufficiently small $R_0$ the following inequality holds
\begin{align*}
	 \A^3  \int  r^{-2\A-4+\epsilon}\W^2 y^a < C	\int r^{-2\A} e^{2\alpha r^{\epsilon}} g^2 y^{-a}, 
\end{align*}
which completes the proof. \end{proof}
\subsection{Vanishing order estimate in the bulk}
Now given that the Carleman estimate \eqref{carleman}  in Lemma \ref{carlmanbulk} holds for all $\A >C(n,s)(1+||V||_{C^{1}}^{1/2s})$, we can now argue  as in the proof of Theorem 1.2 in \cite{BGM}   or as in the proof of Theorem 15  in \cite{Ve}, to assert that the following   quantitative vanishing order estimate in the bulk holds.  We  provide the details for the sake of completeness.
\begin{lemma}\label{bulkdf}
	Let $U$ be a weak solution of \eqref{extpr}. Then there exists $C$ universal such that for all $\rho < R_0/8$, we have
	\begin{align*}
		\int_{\mathbb B_{\rho}\times (0,1)}  U^2 y^a  \ge C \rho^{A},
	\end{align*} 
where $A=C ||V||^{1/2s}_{C^1} +C\left(1+\int_{\mathbb B_{R_0} \times (0,1)}y^aU^2\right)\big/\int_{\mathbb B_{R_0/4}\times (1/4,3/4)}y^aU^2+C$ and $R_0$ is as in Lemma \ref{carlmanbulk}.
\end{lemma}
\begin{proof}
	In the proof, we will denote an all purpose constant by the letter $C$, which might vary from line to line, and will depend only on $n$, $s$, $R_0$ and $T$. For simplicity of the computations and notational convenience we work with the symmetric time-interval $(-T,T)$, instead of $(0,T),$ in our case $T=1.$ For notational convenience, we denote $|X|$ by $r$ and we define the following sets:
	\begin{align*}
		\mathbb A (r_1,r_2)&:=\{(x,y)\in \R^n \times \{y>0\}: r_1 < |X| < r_2\}\\
		A(r_1,r_2)&:=\{(x,0): r_1 < |x| < r_2\}.
	\end{align*}  
	Let $R_0$ be as in the Lemma \ref{carlmanbulk}.  Let $0 < r_1 < r_2/2 < 2r_2= R_0/2$ be fixed, and let $\phi(X)$ be a smooth radial function such that    
	$\phi(X) \equiv 0$ in $\mathbb B_{\frac{
			r_1}{2}} \cup \mathbb B_{2r_2}^c$ and 
	$\phi(X) \equiv 1$ in $\mathbb A (r_1,r_2)$. We now let $T_1= 3T/4$ and $T_2= T/2$, so that $0<T_2<T_1<T$. As in  \cite{Ve}, we let $\eta(t)$ be a smooth even function such that $\eta(t) \equiv 1$ when $|t| < T_2$, $\eta(t) \equiv 0$, when $|t| > T_1$. Furthermore, it will be important in the sequel (see \eqref{bn1} below) that $\eta$  decay exponentially near $t = \pm T_1$. As in (118) of \cite{Ve} we take  
	\begin{equation}\label{d4}
		\eta(t)= \begin{cases} 0\ \ \ \ \ \ \ \ \ -T\le t\le -T_1
			\\
			\exp \left(-\frac{T^3(T_2+t)^4}{(T_1 +t)^3(T_1-T_2)^4} \right)\ \ \ \ \ \ \ -T_1\le t \le -T_2,
			\\
			1,\ \ \ \ \ \ \  t \in -T_2\le t \le 0.
		\end{cases}
	\end{equation} 
	Without loss of generality, we will assume that
	\begin{equation}\label{assume}
		\int_{ \mathbb B_{r_2} \times (-T_2, T_2)} U^2y^a \neq 0.
	\end{equation}
	Otherwise, $U \equiv 0$  in $\mathbb B_{r_2} \times (-T_2, T_2)$, which by Theorem 15 (b) in \cite{Ve}  implies $U\equiv 0$ in $\mathbb B_{R_0}\times (-T_2, T_2)$ and by the assertions that follow we could conclude $U \equiv 0$ in $\mathbb B_{R_0}\times (-T, T).$   
	
	We then let  $\W= \phi \eta U$. Since $U$ is a solution of \eqref{extpr} and $\phi$ is a radial function, we have $\py \phi =0$ on the thin set $\{y=0\},$ therefore  we find that $\W$ solves
	\begin{align*}
		\begin{cases}
			\D(y^a \n \W) + y^a \W_t= 2 \eta y^a \langle\n U, \n \phi \rangle + \D(y^a \n \phi)\eta U + y^a \phi U \eta_t \;\;\; \text{in $\{y>0\}$},
			\\
			\py \W = V\W\;\;\; \text{on $\{y=0\} .$}
		\end{cases}
	\end{align*}
	Since $\eta \phi$ is compactly supported in $(\overline{\mathbb B_{R_0}}\setminus \{0\})\times (-T,T)$, $\W$ is compactly supported in  $(\overline{\mathbb B_{R_0}}\setminus \{0\})\times (-T,T)$. Therefore we apply the Carleman estimate \eqref{carleman} to $\W$, obtaining
	\begin{align*}
		\alpha^3 \int r^{-2\alpha - 4+\epsilon} e^{2\alpha r^\epsilon} \W^2y^a   &\leq C \int r^{-2\alpha}  e^{2\alpha r^\epsilon} (2 \eta y^a \langle\n U, \n \phi\rangle + \D(y^a \n \phi)\eta U + y^a \phi U \eta_t)^2 y^{-a}.
	\end{align*}
	As a consequence of the algebraic inequality $(A+B+C)^2 \le 3( A^2 + B^2 +C^2)$, we get  
	\begin{align}\label{et1}
		\alpha^3 \int r^{-2\alpha - 4+\epsilon} e^{2\alpha r^\epsilon} \W^2 y^a  &\leq 
		C \int r^{-2\alpha}  e^{2\alpha r^\epsilon} ( \eta^2 |\n U|^2 |\n \phi|^2 y^a + (\D(y^a \n \phi))^2\eta^2 U^2y^{-a} +  (\phi U \eta_t)^2y^a).
	\end{align}
	We now recall that the way $\phi$ and $\eta$ have been chosen, $\nabla \phi$ is supported in $\mathbb A\left(\frac{r_1}{2}, r_1\right) \cup \mathbb A\left(r_2, 2r_2\right)$ and in this set
	$|\n \phi(X) | =\frac{|\phi'(|X|)X|}{|X|}= O(1/r)$ and $|\nabla^2 \phi(X) | = O(1/r^2),$ which gives $$ |\D( y^a \n \phi)|= |\Delta \phi y^a+ ay^{a-1} \phi_y|= |\Delta \phi y^a+ ay^{a-1}(y \phi')/r|=y^a O(1/r^2),$$ hence \eqref{et1} becomes
	\begin{align}\label{et22}
		& \alpha^3 \int r^{-2\alpha- 4+\epsilon} e^{2\alpha r^\epsilon} \W^2 y^a 
		\leq C \int_{ \mathbb A \left(\frac{r_1}{2}, r_1\right) \times (-T_1,T_1)  } e^{2\alpha r^\epsilon} ( r^{-2\alpha-2} |\n U|^2 y^a + r^{-2\alpha-4}  U^2y^a)
		\\
		&   + C  \int_{  \mathbb A \left(r_2, 2r_2\right) \times (-T_1, T_1)  } e^{2\alpha r^\epsilon} ( r^{-2\alpha-2} |\n U|^2 y^a + r^{-2\alpha-4}  U^2y^a) 
		\notag\\
		& + C \int_{ \mathbb A \left(\frac{r_1}{2}, 2r_2\right) \times (-T_1, T_1) } r^{-2\alpha} e^{2\alpha r^\epsilon}    \phi^2  \eta_t^2U^2y^a.
		\notag
	\end{align}
	Since the functions 
	\begin{equation}\label{monotone}
		r \to r^{-2\alpha-4} e^{2\alpha r^\epsilon}, \ \ \ \ \ 	r \to r^{-2\alpha-2} e^{2\alpha r^\epsilon}\ \ \ \ \ \ \text{and}\  \ \ \ \ \ \  r \to r^{-2\alpha} e^{2\alpha r^\epsilon}  
	\end{equation}
	are decreasing in $(0,1)$, we  can estimate the first and second integral in the right hand side of  \eqref{et22} in the following way
	\begin{align}\label{r11}
		& \int_{  \mathbb A \left(\frac{r_1}{2},r_1\right) \times (-T_1, T_1)  } e^{2\alpha r^\epsilon} ( r^{-2\alpha-2} |\n U|^2 y^a + r^{-2\alpha-4}  U^2y^a)\\ 
		&\leq C \left(\frac{r_1}{2}\right)^{-2\alpha-2} e^{\frac{2\alpha r_1^\epsilon}{2^{\epsilon}}}\int_{  \mathbb A \left(\frac{r_1}{2},r_1\right) \times (-T_1, T_1)  } |\nabla U|^2 y^a + C \left(\frac{r_1}{2}\right)^{-2\alpha-4} e^{\frac{2\alpha r_1^\epsilon}{2^{\epsilon}}} \int_{  \mathbb A \left(\frac{r_1}{2},r_1\right) \times (-T_1, T_1)  }  U^2 y^a\notag
	\end{align}
	and
	\begin{align}\label{r22}
		&\int_{  \mathbb A (r_2,2r_2) \times (-T_1, T_1)  } e^{2\alpha r^\epsilon} ( r^{-2\alpha-2} |\n U|^2 y^a + r^{-2\alpha-4}  U^2y^a)\\ 
		&\leq C r_2^{-2\alpha-2} e^{2\alpha r_2^\epsilon} \int_{  \mathbb A (r_2,2r_2) \times (-T_1, T_1)  } |\nabla U|^2 y^a + C r_2^{-2\alpha-4} e^{2\alpha r_2^\epsilon} \int_{  \mathbb A (r_2,2r_2) \times (-T_1, T_1)  }  U^2 y^a.\notag
	\end{align}
	We then note that the following  energy estimate holds
	\begin{align}\label{cc1}
		\int_{ \mathbb{A}\left(\frac{r_1}{2},r_1\right) \times (-T_1, T_1)  } |\n U|^2 y^a \le \frac{C\A^2}{r_1^2}\int_{ \mathbb{A}\left(\frac{r_1}{4},\frac{3r_1}{2}\right) \times (-T, T)  } U^2 y^a.
	\end{align}
	\eqref{cc1} follows from  the energy estimate in the proof of  Theorem 5.1   in \cite{BG}  and also by using that $\alpha = C(||V||_1^{1/2s}+1)$.

	Similarly we find
	\begin{align}\label{cc2}
		\int_{  \mathbb{A}\left(r_2,2r_2\right) \times (-T_1, T_1)  } |\n U|^2 y^a \le \frac{C\A^2}{r_2^2}\int_{  \mathbb{A}{\left(\frac{r_2}{2},4r_2\right)} \times (-T, T)  } U^2 y^a.
	\end{align}
	Using  \eqref{cc1} in \eqref{r11} we get for some universal $\tilde C$ that the following inequality holds
	\begin{align}\label{r12}
		C  \int_{  \mathbb A \left(\frac{r_1}{2},r_1\right) \times (-T_1, T_1)  } e^{2\alpha r^\epsilon} ( r^{-2\alpha-2} |\n U|^2 y^a + r^{-2\alpha-4}  U^2y^a) 
		\leq \tilde{C}\A^2 \left(\frac{r_1}{2}\right)^{-2\alpha-4} e^{\frac{2\alpha r_1^\epsilon}{2^{\epsilon}}}\int_{  \mathbb A \left(\frac{r_1}{4},\frac{3r_1}{2}\right) \times (-T, T)  } U^2 y^a. 
	\end{align}
Similarly, using \eqref{cc2} in \eqref{r22} we obtain	\begin{align}\label{r23}
		C  \int_{  \mathbb A (r_2,2r_2) \times (-T_1, T_1)  } e^{2\alpha r^\epsilon} ( r^{-2\alpha-2} |\n U|^2 y^a + r^{-2\alpha-4}  U^2y^a)\leq \tilde{C}\A^2 r_2^{-2\alpha-4} e^{2\alpha r_2^\epsilon} \int_{  \mathbb A \left(\frac{r_2}{2}, 4r_2\right) \times (-T, T)  }  U^2 y^a.
	\end{align}
	Thus we have estimated the first and second term in the right-hand side of \eqref{et22}. We now estimate the last term in the right-hand side of \eqref{et22}. We start by spliting the last term in the right-hand side of \eqref{et22} as follows
	\begin{align}\label{spt}
		& C \int_{ \mathbb A \left(\frac{r_1}{2},2r_2\right) \times (-T_1,T_1) } r^{-2\alpha} e^{2\alpha r^\epsilon}    \phi^2  \eta_t^2 U^2 y^a= C \int_{ \mathbb A \left(\frac{r_1}{2},r_1\right) \times (-T_1,T_1)} r^{-2\alpha} e^{2\alpha r^\epsilon}   \phi^2 \eta_t^2 U^2y^a
		\\
		& + C \int_{ \mathbb A \left(r_1,r_2\right) \times (-T_1,T_1)} r^{-2\alpha}  e^{2\alpha r^\epsilon}  \phi^2 \eta_t^2U^2y^a + C \int_{ \mathbb A \left(r_2,2r_2\right) \times (-T_1, T_1)} r^{-2\alpha} e^{2\alpha r^\epsilon}   \phi^2 \eta_t^2 U^2y^a.
		\notag
	\end{align}
	Using $|\eta_t| \leq C/T$, $\phi \leq 1$ and \eqref{monotone}, we observe that the first term of the right hand side of \eqref{spt} can be estimated as
	\begin{align}\label{g1}
		C \int_{\mathbb A \left(\frac{r_1}{2},r_1\right) \times (-T_1,T_1)} r^{-2\alpha} e^{2\alpha r^\epsilon}   \phi^2 \eta_t^2 U^2y^a &\leq  C \left(\frac{r_1}{2} \right)^{-2\alpha} e^{ \frac{2 \alpha r_1^\epsilon}{2^\epsilon}} \int_{\mathbb A \left(\frac{r_1}{2},r_1\right)  \times (-T_1,T_1)}  U^2y^a\\
		&\leq C \left( \frac{r_1}{2C_1}\right) ^{-2\alpha}\int_{\mathbb A \left(\frac{r_1}{2},r_1\right)  \times (-T_1,T_1)}  U^2y^a\notag\end{align}
	where $C_1=e^{R_0^\epsilon}$, which is a consequence of the facts that the exponential function is a increasing function and $r_1/2 <R_0.$ Since $R_0$ is a universal constant, $C_1$ is a universal constant. Similarly, the last term  in the right hand side of \eqref{spt} can be upper bounded in the following way
	\begin{equation}\label{g2}
		C \int_{\mathbb A (r_2,2r_2) \times (-T_1,T_1)} r^{-2\alpha} e^{2\alpha r^\epsilon}  \phi^2 \eta_t^2 U^2 y^a \leq  C r_2^{-2\alpha} e^{2\alpha r_2^\epsilon}  \int_{\mathbb A (r_2,2r_2) \times (-T_1,T_1)}  U^2y^a.
	\end{equation}
	Now by employing a fairly deep  idea  in the proof of Theorem 15 in \cite{Ve} ( see pages 659-661 in \cite{Ve} or  the proof of (3.12) in \cite{BGM}), we can assert that the following estimate holds for the second term in the right-hand side of \eqref{spt} 	\begin{align}\label{vc}
		&C\int_{\mathbb A (r_1,r_2) \times (-T_1,T_1)}  r^{-2\alpha} e^{2\alpha r^\epsilon}   \eta_t^2 \phi^2 U^2 y^a \leq \frac{\A^3}{2}\int_{\mathbb A (r_1,r_2) \times (-T_1,T_1)}  r^{-2\alpha-4 +\epsilon} e^{2\alpha r^\epsilon}   \W^2 y^a+  C \int_{\mathbb B_{R_0 } \times (-T,T)} U^2y^a. 
	\end{align} 
	Over here we would like to emphasize that the presence of the coefficient $\alpha^3$ in front of the integral $\int_{\mathbb A (r_1,r_2) \times (-T_1,T_1)}  r^{-2\alpha-4 +\epsilon} e^{2\alpha r^\epsilon}   \W^2 y^a$ plays a crucial role in the derivation of  the estimate \eqref{vc}.  Therefore  we see that  the presence of the coefficient  $\alpha^3$  in front of the same  integral in \eqref{carleman} is very crucial for  deducing the vanishing order estimate from \eqref{carleman}. This is indeed a new feature in the parabolic case.
	
	Now using the estimates \eqref{g1}, \eqref{g2} and \eqref{vc} in \eqref{spt}, we get
	\begin{align}\label{spt1}
		&C \int_{ \mathbb A \left(\frac{r_1}{2},2r_2\right) \times (-T_1,T_1) } r^{-2\alpha} e^{2\alpha r^\epsilon}    \phi^2  \eta_t^2 U^2 y^a\\
		&\le  C \left( \frac{r_1}{2C_1}\right) ^{-2\alpha}\int_{\mathbb A \left(\frac{r_1}{2},r_1\right)  \times (-T_1,T_1)}  U^2y^a +C r_2^{-2\alpha} e^{2\alpha r_2^\epsilon}  \int_{\mathbb A (r_2,2r_2) \times (-T_1,T_1)}  U^2y^a\notag\\
		&\;\;+\frac{\A^3}{2}\int_{\mathbb A (r_1,r_2) \times (-T_1,T_1)}  r^{-2\alpha-4 +\epsilon} e^{2\alpha r^\epsilon}   \W^2 y^a+  C \int_{\mathbb B_{R_0} \times (-T,T)} U^2y^a.\notag
	\end{align}
Thus we have estimated all the three terms in the right-hand side of \eqref{et22}. 
	Therefore using  \eqref{r12}, \eqref{r23} and \eqref{spt1} in \eqref{et22}, we find 
	\begin{align}\label{et02}
		&\A^3\int_{\mathbb A (r_1,r_2) \times (-T_1,T_1)}  r^{-2\alpha-4 +\epsilon} e^{2\alpha r^\epsilon}   \W^2 y^a\\
		&\le C\A^2 \left(\frac{r_1}{2}\right)^{-2\alpha-4} e^{\frac{2\alpha r_1^\epsilon}{2^{\epsilon}}}\int_{  \mathbb A \left(\frac{r_1}{4},\frac{3r_1}{2}\right) \times (-T, T)  } U^2 y^a +C\A^2 r_2^{-2\alpha-4} e^{2\alpha r_2^\epsilon} \int_{  \mathbb A \left(\frac{r_2}{2},4r_2\right) \times (-T, T)  }  U^2 y^a\notag\\
		& \;\;+ C \left( \frac{r_1}{2C_1}\right) ^{-2\alpha}\int_{\mathbb A \left(\frac{r_1}{2},r_1\right)  \times (-T_1,T_1)}  U^2y^a +C r_2^{-2\alpha} e^{2\alpha r_2^\epsilon}  \int_{\mathbb A \left(r_2, 2r_2\right) \times (-T_1,T_1)}  U^2y^a\notag\\
		&\;\;+\frac{\A^3}{2}\int_{\mathbb A (r_1,r_2) \times (-T_1,T_1)}  r^{-2\alpha-4 +\epsilon} e^{2\alpha r^\epsilon}   \W^2 y^a+  C \int_{\mathbb B_{R_0} \times (-T,T)} U^2y^a.\notag
	\end{align}
	We first observe that the term $\frac{\A^3}{2}\int_{\mathbb A (r_1,r_2) \times (-T_1,T_1)}  r^{-2\alpha-4 +\epsilon} e^{2\alpha r^\epsilon}   \W^2 y^a$ can be absorbed in the left hand side of \eqref{et02}. 	Over here we also remark that since $\tilde{C}$ in  \eqref{r12}, \eqref{r23} are universal, therefore it has been replaced by $C$ in \eqref{et02} above.
	Furthermore, since $r_1/2 <R_0,$ and $C_1=e^{R_0^{\epsilon}}$ we have
	\begin{align}\label{v1}
		 \left(\frac{r_1}{2}\right)^{-2\alpha-4} e^{\frac{2\alpha r_1^\epsilon}{2^{\epsilon}}} \le \left(\frac{r_1}{2C_1}\right)^{-2\alpha-4}.
	 \end{align} 
 Also we note \begin{align}\label{v2}
 	\mathbb A \left(\frac{r_1}{2},r_1\right) \subset \mathbb A \left(\frac{r_1}{4},\frac{3r_1}{2}\right)\;\;\; \text{and}\;\;\;\mathbb A (r_2,2r_2) \subset \mathbb A (r_2/2,4r_2) \subset \mathbb B_{R_0}.
 \end{align}
 Therefore using \eqref{v1} and \eqref{v2} in\eqref{et02}, we obtain
	\begin{align}\label{et4}
		&\frac{\A^3}{2}\int_{\mathbb A (r_1,r_2) \times (-T_1,T_1)}  r^{-2\alpha-4 +\epsilon} e^{2\alpha r^\epsilon}   \W^2 y^a\\
		&\le C(\A^2+1) \left(\frac{r_1}{2C_1}\right)^{-2\alpha-4} \int_{  \mathbb A \left(\frac{r_1}{4},\frac{3r_1}{2}\right) \times (-T, T)  } U^2 y^a +C(\A^2+2) r_2^{-2\alpha-4} e^{2\alpha r_2^\epsilon} \int_{  \mathbb B_{R_0} \times (-T, T)  }  U^2 y^a\notag
	\end{align}
	
	Using \eqref{monotone},  the fact that $\eta\equiv 1$ and $\phi \equiv 1$ in $\mathbb A(r_1,r_2) \times (-T_2,T_2),$  we find that the integral in the left-hand side of \eqref{et4} can be bounded from below in the following way
	\begin{equation}\label{b5}
		\frac{\alpha^3}{2} \int r^{-2\alpha- 4+\epsilon} e^{2\alpha r^\epsilon} \W^2y^a \geq   \frac{\alpha^3}{2} r_2^{-2\alpha - 4+\epsilon}  e^{2\alpha r_2^\epsilon} \int_{\mathbb A(r_1,r_2) \times (-T_2,T_2)} U^2y^a. 
	\end{equation}
	Subsequently using the bound \eqref{b5} in \eqref{et4},and then  by dividing both sides  of the resulting inequality in \eqref{et4}  by $\A^2 
	r_2^{-2\alpha - 4+ \epsilon} e^{2\alpha r_2^\epsilon}$, we obtain
	\begin{align}\label{e3}
		\frac{\alpha}{2} \int_{\mathbb A(r_1,r_2) \times (-T_2,T_2)} U^2y^a \leq C \left(\frac{r_1}{2C_1r_2}\right)^{-2\alpha -4}r_2^{-\epsilon}  \int_{ \mathbb A\left(\frac{r_1}{4}, \frac{3r_1}{2}\right) \times (-T, T)  }  U^2y^a + Cr_2^{-\epsilon}  \int_{ \mathbb B_{R_0} \times (-T, T)  } U^2y^a.
	\end{align}
	 Since $r_2(=R_0/4)$ and $\epsilon$ are  universal constants, we can replace $Cr_2^{-\epsilon}$ by a new universal $C.$  We  now add $\frac{\alpha}{2} \int_{ \mathbb B_{r_1} \times (-T_2,T_2)} U^2y^a$ to both sides of the inequality \eqref{e3} to deduce the following
	\begin{align}\label{et5}
		\frac{\alpha}{2} \int_{\mathbb B_{r_2} \times (-T_2, T_2)} U^2y^a
		& \leq   C \left(\frac{r_1}{2C_1r_2}\right)^{-2\alpha -4}  \int_{  \mathbb A\left(\frac{r_1}{4}, \frac{3r_1}{2}\right)  \times (-T, T)}  U^2y^a
		+ \frac{\alpha}{2} \int_{\mathbb B_{r_1} \times (-T_2, T_2)} U^2y^a \\
		&\;\;\;\;+ C   \int_{ \mathbb B_{R_0} \times (-T, T)  } U^2y^a\notag\\
		& \le 2 C \left(\frac{r_1}{2C_1r_2}\right)^{-2\alpha - 4}   \int_{ \mathbb B_{3r_1/2}  \times (-T, T)  }  U^2y^a
		+C  \int_{ \mathbb B_{R_0} \times (-T, T)  } U^2y^a, \notag
	\end{align}
	where in the second inequality in \eqref{et5}, we have used, 
	\begin{equation}\label{lkjh}
	\frac{\alpha}{2} \leq \left(\frac{r_1}{2C_1r_2}\right)^{-2\alpha - 4}.
	\end{equation}
	\eqref{lkjh} can be seen as follows. Since $2r_1 <r_2$ and $C_1>1$, therefore $$\left(\frac{2C_1 r_2}{r_1}\right)^{2\A+4} \ge (4C_1)^{2\A+4} \ge \frac{\A}{2}.$$
	Using  \eqref{assume}, we now choose $\alpha$ (depends also on $U$) such that
	\begin{equation}\label{choice2}
		\frac{\alpha}{2} \int_{ \mathbb B_{r_2} \times (-T_2,T_2)} U^2 y^a  \geq  1+   C  \int_{ \mathbb B_{R_0} \times (-T, T)  } U^2y^a. 
	\end{equation}
Using \eqref{choice2} in \eqref{et5}, we get
\begin{align}\label{o1}
	 1+   C  \int_{ \mathbb B_{R_0} \times (-T, T)  } U^2y^a \leq 2 C \left(\frac{r_1}{2C_1r_2}\right)^{-2\alpha - 4}   \int_{ \mathbb B_{3r_1/2}  \times (-T, T)  }  U^2y^a
	 +C  \int_{ \mathbb B_{R_0} \times (-T, T)  } U^2y^a.
\end{align}
Now  \eqref{o1} can be  equivalently  written as 
	\begin{align}\label{et7}
		\left(\frac{r_1}{C_1 r_2}\right)^{2 \alpha + 4}   \leq 2 C \int_{\mathbb B_{3r_1/2} \times (-T, T) } U^2y^a.
	\end{align}
	Let $\rho=3r_1/2,$ then \eqref{et7} becomes 
	\begin{align}\label{et8}
		\rho^{2\A+4}   \leq 2 C \left(\frac{3 C_1 r_2}{2}\right)^{2 \alpha + 4}      \int_{\mathbb B_{\rho} \times (-T, T) } U^2y^a.
	\end{align}
We now multiply both sides of the inequality  \eqref{et8} by $\rho^{2\A+4}$ and noting that since $\rho <R_0/8$ and $C_1=e^{R_0^{\epsilon}}<e$, it follows that \begin{equation}\label{plkj}\frac{3C_1 r_2 \rho}{2} \le \frac{3e R_0^2}{64} \le 1. \end{equation}  Using \eqref{plkj} to the resulting inequality  in \eqref{et8} we obtain
	\begin{align}\label{et9}
		\rho^{4\A+8}     \leq 2 C\left(\frac{3 C_1 r_2\rho}{2}\right)^{2 \alpha + 4} \int_{\mathbb B_{\rho} \times (-T, T) } U^2y^a  \leq 2 C \int_{\mathbb B_{\rho} \times (-T, T) } U^2y^a.
	\end{align}  
	Now by letting
	\begin{align}\label{alp}
		\alpha= C(n,s)(||V||_{C^1}^{1/2s}+1) +2C\frac{1+\int_{ \mathbb B_{R_0} \times (-T, T)  } U^2y^a}{\int_{ \mathbb B_{r_2} \times (-T_2,T_2)} U^2 y^a}+\overline{C},
	\end{align}
	 which guarantees that \eqref{choice2} holds, we can now assert that there  exists a universal constant $ C$ such that the following holds for all $0< \rho \leq R_0/8$,
	\[
	\int_{\mathbb B_{\rho} \times (-T, T) } U^2 y^a\geq  C \rho^{A}.
	\]
	where $A= C ||V||_{C^1}^{1/2s} + C \frac{1+\int_{ \mathbb B_{R_0} \times (-T, T)  } U^2y^a}{\int_{ \mathbb B_{R_0/4} \times (-T_2,T_2)} U^2 y^a}+  C.$ This completes the proof of Lemma \ref{bulkdf}.
\end{proof}

\subsection{Carleman estimate II} Our next Carleman estimate is a parabolic generalization of the estimate in \cite[Proposition 5.7]{RS}.  As the reader will find, this estimate allows passage of quantitative uniqueness from the bulk to the boundary.  This is precisely where we require that $s \in [1/2, 1)$ similar to that in \cite{RS}.
\begin{lemma}\label{carlbdr}
	Let $\W$ be a  solution to
	\begin{align*}
		\D(y^a \n \W) + y^a \W_t&=f \hspace{2mm}\text{in} \hspace{2mm} \{y>0\}\\
		\W&=0\hspace{2mm} \text{on}\hspace{2mm} \{y=0\},
	\end{align*}
	with supp$(\W) \subset  \overline{ \mathbb B_{1/2}}\times (0,1).$ Assume $||y^{-a/2}f||_{L^2} < \infty$ and also assume that $\nabla_x \W, \W_t$ are continuous up to the thin set $\{y=0\}.$ Then  there exist a universal $C=C(n,s)$ such that for any $\A >2,$ we have 
	\begin{align}\label{carlbdrest}
		\A^3 \int e^{2\A \phi}\W^2y^a \le C \int e^{2\A \phi}f^2y^a +C \A \int_{\{y=0\}} (\py \W)^2.
	\end{align}
\end{lemma}
\begin{proof}
	Similar to that in \cite{RS}, we let  $$\phi(x,y)=-\frac{|x|^2}{4} + \gamma \left[-\frac{y^{a+1}}{a+1}+\frac{y^2}{2}\right].$$
	We then set $W=e^{\A \phi}y^{a/2}\W$. Thus we have  $\W=e^{-\A \phi}y^{-a/2}W$. Now by direct calculation,  we find
	\begin{align*}
		e^{\A \phi}y^{-a/2}\D(y^a \n \W) +	e^{\A \phi}y^{a/2}\W_t &=\Delta W + \A^2 (|x|^2/4 +\gamma^2 (y^{2a}+y^2-2y^{a+1}))W -\A (-n/2-\gamma a y^{a-1}+ \gamma)W\\
		&\;\;\;+ \A \langle \n_x W, x \rangle -2 \A \gamma(y- y^a)W_y+ c_sWy^{-2}+ W_t,
	\end{align*}
	where $$ c_s:=-\frac{a^2}{2} +\frac{a}{2} + \frac{a^2}{4}=-\frac{a^2}{4} +\frac{a}{2}=\frac{1-4s^2}{4}.$$
	Before proceeding further, we mention that in the ensuing computations below, all  the integrals will be on the set  $\{y>\epsilon\}$ but the domain of the integration will not be specified for the simplicity of exposition.  In the end,  we will take the limit $\epsilon \rightarrow 0$.

	We  begin by applying the algebraic inequality $(A+B)^2 \ge A^2 +2 A B$ to
	\begin{align*}
		A&=\Delta W + \A^2 (|x|^2/4 +\gamma^2 (y^a-y)^2)W + c_s Wy^{-2},\\
		B&=-\A (-n/2-\gamma a y^{a-1}+ \gamma)W	+ \A \langle \n_xW, x \rangle -2 \A \gamma(y- y^a)W_y + W_t,
	\end{align*}
	and consequently obtain
	\begin{align}\label{bb1}
		\int e^{2\A\phi}f^2 y^{-a} \ge &\int (\Delta W + \A^2 (|x|^2/4 +\gamma^2 (y^a-y)^2)W + c_s Wy^{-2}
		)^2 -2\A \int  (-n/2-\gamma a y^{a-1}+ \gamma)W\Delta W\\
		&+2\A\int \langle \n_xW, x \rangle \Delta W -4 \A \gamma\int (y- y^a)W_y\Delta W  \notag \\
		&+2  \int (\A^3 |x|^2/4 +\A^3\gamma^2 (y^a-y)^2+ \A c_s y^{-2})\langle \n_x W, x \rangle W \notag \\
		&-4\A\gamma \int  (\A^2|x|^2/4 +\A^2\gamma^2 (y^a-y)^2 + c_s y^{-2} )(y-y^a)WW_y \notag \\
		& + 2\int W_t \Delta W + 2  \int (\A^2|x|^2/4 +\A^2\gamma^2 (y^a-y)^2+c_sy^{-2})WW_t \notag\\
		& -2 \int (\A^3|x|^2/4 +\A^3\gamma^2 (y^a-y)^2+\A c_s y^{-2} )(-n/2-\gamma a y^{a-1}+ \gamma)W^2  \notag\\
		&=I_1+I_2+I_3+I_4+I_5+I_6+I_7+I_8+I_9.\notag
	\end{align}
	We estimate each of the  integrals above separately. First  we note that by using an argument as in Step 3 in the proof of Proposition 5.7 in \cite{RS}, we  have	\begin{align}\label{ftc}
		\int_{\{y=\epsilon\}}e^{2\A \phi}y^{2(a-1)}\W^2 \le C\int_{\{y=0\}}e^{2\A \phi}(\py \W)^2\ \text{as $\epsilon \to 0$}. 
	\end{align}
	The inequality \eqref{ftc} will be crucially used in the estimate for $I_j$'s  below.
	Using integration by parts and also using  the fact that $\text{supp}(\W) \subset \overline{\mathbb B_{1/2}} \times(0,1)$, we get
	\begin{align*}
	 I_2=	-2\A \int (-n/2-\gamma a y^{a-1}+ \gamma)W \Delta W &=  2\A \int (\gamma -\gamma a y^{a-1} -n/2) |\n W|^2 -2\A \int \gamma a (a-1) y^{a-2}WW_y\\
		&\;\;+2 \A \int_{\{y=\epsilon\}}(-n/2-\gamma a y^{a-1}+ \gamma)WW_y
	\end{align*}
	Another application of integration by parts gives
	\begin{align}\label{thr}
		&-2\A \int \gamma a (a-1) y^{a-2}WW_y=-\A \int \gamma a (a-1) y^{a-2}\partial_y W^2\\
		&=\A \int \gamma a (a-1)(a-2) y^{a-3} W^2 + \A \int_{\{y=\epsilon\}} \gamma a (a-1) y^{a-2}W^2\notag\\
		& \geq \A \int \gamma a (a-1)(a-2) y^{a-3} W^2.\notag	\end{align}
		In the last inequality in \eqref{thr} above, we used that  since we are  in the $s \geq \frac{1}{2}$, therefore we have that $a(a-1) \geq 0$ which in particular implies that
		\[
		\A \int_{\{y=\epsilon\}} \gamma a (a-1) y^{a-2}W^2\ge 0.
		\]		
		
		 Now on substituting $W=e^{\A\phi}y^{a/2}\W$, we find that
	$W_y= e^{\A \phi}(y^{a/2}\W_y +\frac{a}{2}y^{a/2-1}\W+ \A(-\gamma y^a + \gamma y)y^{a/2}\W)$ and thus we get  for small enough $\epsilon>0$
\begin{align}\label{htu}
		&2 \A \int_{\{y=\epsilon\}}(-n/2-\gamma a y^{a-1}+ \gamma)WW_y\\
		&=2 \A \int_{\{y=\epsilon\}}(-n/2-\gamma a y^{a-1}+ \gamma)e^{2\A \phi}(y^{a}\W_y +\frac{a}{2}y^{a-1}\W+ \A(-\gamma y^a + \gamma y)y^{a}\W)\W\notag\\
		&\geq - 4\A\gamma\int_{\{y=\epsilon\}} e^{2\alpha \phi}|y^{a-1}\W||y^a\W_y| - 2 \A \gamma \int_{\{y=\epsilon\}}e^{2\A \phi} y^{2(a-1)}\W^2 - 4\A^2\gamma^2 \int_{\{y=\epsilon\}} e^{2\A \phi}y^{2(a-1)}y^{1+a}\W^2\notag\\
		& \geq- 6\A\gamma C\int_{\{y=0\}} e^{2\alpha \phi} (\py \W)^2\ \text{as $\epsilon \to 0$}.\notag 
	\end{align}
	In the first inequality  in \eqref{htu} above, we used that $a- 1<0$ and thus for small enough $y$, we have $y^{a-1}>>1$.
	In the second inequality, we used \eqref{ftc} to estimate the second and the third term on the right hand side of the first inequality in \eqref{htu}  in the following way
	\begin{equation}
	\begin{cases}
	2 \A \gamma \int_{\{y=\epsilon\}}e^{2\A \phi} y^{2(a-1)}\W^2 \leq 2 \alpha \gamma C \int_{\{y=0\}} e^{2\alpha \phi} (\py \W)^2, \text{as $\epsilon \to 0$},
	\\
	4\A^2\gamma^2 \int_{\{y=\epsilon\}} e^{2\A \phi}y^{2(a-1)}y^{1+a}\W^2 \to 0\ \text{as $\epsilon \to 0$}.	
	\end{cases}
	\end{equation}
	We thus obtain as $\epsilon \to 0$
	\begin{align}\label{i2}
	I_2=	-2\A \int (-n/2-\gamma a y^{a-1}+ \gamma)W \Delta W&\ge 2\A \int (\gamma -\gamma a y^{a-1} -n/2) |\n W|^2 +\A \int \gamma a (a-1)(a-2) y^{a-3} W^2\\
		&\;\;\;- 6C\A\gamma\int_{\{y=0\}} e^{2\alpha \phi} (\py \W)^2.\notag
	\end{align}
	Now to estimate $I_3$, we  first apply Rellich identity (Lemma \ref{Rellich}) for $a=0$ and $k=0$ to get
	\begin{align}\label{ter1}
		2 \A\int \Delta_x W \langle \n_x W,x \rangle
		=(n-2)\A \int |\n_x W|^2.
	\end{align}
	Furthermore, by integrating by parts we observe
	\begin{align}\label{ter2}
		2 \A \int W_{yy}\langle \n_x W,x \rangle&=-2\A \int W_y \langle  \n_x W_y,x\rangle -2 \A \int_{\{y=\epsilon\}}W_y \langle  \n_x W,x\rangle \\
		&=-\A \int \langle  \n_x W_y^2,x\rangle -2 \A \int_{\{y=0\}}W_y \langle  \n_x W,x\rangle\notag\\
		&=n\A \int W_y^2 -2 \A \int_{\{y=\epsilon\}}W_y \langle  \n_x W,x\rangle.\notag
	\end{align}
	Hence from \eqref{ter1} and \eqref{ter2} we have
	\begin{align}\label{ter4}
	I_3=	2 \A\int \Delta W \langle \n_x W,x \rangle=(n-2)\A \int |\n_x W|^2+n\A \int W_y^2 -2 \A \int_{\{y=0\}}W_y \langle  \n_x W,x\rangle.
	\end{align}
In order to the estimate the boundary integral in \eqref{ter4} above, we substitute  $W$ in terms of  $\W$ and then by computations  as in \eqref{htu} above, we get
	\begin{align}\label{ter5}
	&-2\A \int_{\{y=\epsilon\}}W_y \langle x, \n_x W \rangle \\
	&=-2\A \int_{\{y=\epsilon\}}e^{2\A \phi}(y^{a}\W_y +\frac{a}{2}y^{a-1}\W+ \A(-\gamma y^a + \gamma y)y^{a}\W)(\langle x, \n_x \W\rangle -\frac{\A}{2} |x|^2\W )\notag\\
	& =- 2 \A \int_{\{y=\epsilon\}}e^{2\A \phi}y^{a}\W_y\langle x, \n_x \W\rangle - \A a \int_{\{y=\epsilon\}}e^{2\A \phi}y^{a-1}\W\langle x, \n_x \W\rangle +2\A^2 \gamma \int_{\{y=\epsilon\}}e^{2\A \phi}y^{1+a}y^{a-1}\W\langle x, \n_x \W\rangle\notag\\
	&+\A^2 \int_{\{y=\epsilon\}}e^{2\A \phi}|x|^2y^{1-a}(y^{a-1}\W)(y^a \W_y) +\frac{\A^2a}{2} \int_{\{y=\epsilon\}}e^{2\A \phi}y^{1-a} |x|^2(y^{a-1}\W)^2 -\alpha^3 \gamma \int_{\{y=\epsilon\}}e^{2\A \phi}y^2(y^{a-1}\W)^2 \notag\\
	&- 2\alpha^2 \gamma \int_{\{y=\epsilon\}} e^{2\alpha \phi} y^{2}y^{a-1} \W <x, \nabla_x \W> +\alpha^3 \gamma \int_{\{y=\epsilon\}} y^{1+a} |x|^2 \W^2.\notag
	\end{align}
	Now  by an application of Cauchy-Schwartz, \eqref{ftc} and the fact that $\W, <\nabla_x \W, x> \to 0$ as $\epsilon \to 0$( recall $\W \equiv 0$ at $\{y=0\}$), we observe that  all the integrals in \eqref{ter5} go  to $0$ as $ \epsilon \to 0$. We thus conclude that as $\epsilon \to 0,$
\begin{align}\label{i3}
I_3 \ge 	2 \A\int \Delta W \langle \n_x W,x \rangle=(n-2)\A \int |\n_x W|^2+n\A \int W_y^2.
\end{align}
We now estimate $I_4$. We have 	\begin{align*}
	I_4=	-4 \A \gamma \int \Delta W (y-y^a)W_y&=4 \A \gamma \int \langle \n_x W,(y-y^a)\n_x W_y \rangle - 4 \A \gamma \int W_{yy} W_y(y-y^a)\\
		&= 2 \A \gamma \int (y-y^a)\partial_y |\n_x W|^2  - 2 \A \gamma \int (W_y^2)_y(y-y^a)\\
		&\ge-2\A \gamma \int (1-ay^{a-1})|\n_x W|^2
		+2\A \gamma \int (1-ay^{a-1})W_y^2 \\ &\;\;+2 \A\gamma \int_{\{y=\epsilon\}}(y-y^a)W_y^2.
	\end{align*}
We now estimate the boundary integral in the above expression. We find
\begin{align*}
	\left|2\A \gamma \int_{\{y=\epsilon\}}(y- y^a) W_y^2 \right| &\le  2\A \gamma \int_{\{y=\epsilon\}}(y^a)e^{2\A \phi}(y^{a/2}\W_y +\frac{a}{2}y^{a/2-1}\W+ \A(-\gamma y^a + \gamma y)y^{a/2}\W)^2\\
	&\le 4\A\gamma \int_{\{y=\epsilon\}} e^{2\A \phi} [(y^a\W_y)^2+ (y^{a-1}\W)^2 +\A y^{2a+2}(y^{a-1}\W)^2]\\
	&\le 8C\A\gamma \int_{\{y=0\}} e^{2\A \phi} (\py \W)^2\ \text{as $\epsilon \to 0$.}
\end{align*}
Note that in the last inequality, we again used \eqref{ftc}. 
Hence we get as $\epsilon \to 0,$
 \begin{align}\label{i4}
	I_4 \ge -2\A \gamma \int (1-ay^{a-1})|\n_x W|^2
	+2\A \gamma \int (1-ay^{a-1})W_y^2 -8C\A\gamma \int_{\{y=0\}} e^{2\A \phi} (\py \W)^2.
\end{align}
To estimate $I_5$, we  first write $W\n_x W =\frac{1}{2} \n_x W^2$ and then  integrate by parts with respect to $x$-variable to get
	\begin{align}\label{i5}
	I_5=	2& 	\int  (\A^3|x|^2/4 +\A^3\gamma^2 (y^a-y)^2+\A c_s y^{-2})W\langle \n_x W,x \rangle\\
		&= -\frac{\A^3}{4} \int (n+2)|x|^2W - n \A^3 \gamma^2 \int (y^a-y)^2W^2 -n \A c_s \int  W^2 y^{-2}.\notag
	\end{align}
	Similarly, by writing $ 2 W W_y = \partial_y W^2$ and  then by integrating by parts in the $y$-variable, we have
	\begin{align*}
	&I_6=	-4\A\gamma \int  (\A^2|x|^2/4 +\A^2\gamma^2 (y^a-y)^2 + c_s y^{-2} )(y-y^a)WW_y\\
		&=\frac{\A^3}{2} \gamma \int  (1-ay^{a-1})|x|^2W^2
		+6 \A^3\gamma^3 \int (y-y^a)^2(1-ay^{a-1})W^2\\
		&+2\A^3 \gamma \int_{\{y=\epsilon\}}  (|x|^2/4 +\gamma^2 (y-y^a)^2)(y-y^a)W^2\\
		&-2 \A \gamma c_s \int (y^{-2}+(a-2)y^{a-3})W^2 + 2 \A \gamma c_s \int_{\{y=\epsilon\}} (y-y^a)y^{-2}W^2.	\end{align*}
Again the first boundary integral in the expression of $I_6$ above can be estimated by writing $W$ in terms of $\W$ and by using \eqref{ftc} in the following way\begin{align*}
	2\A^3 \gamma \int_{\{y=\epsilon\}}  (|x|^2/4 +\gamma^2 (y-y^a)^2)(y-y^a)W^2 &\le C \A^3 \gamma^3 \int_{\{y=\epsilon\}} e^{2 \A \phi }y^{4a}\W^2\\
	& = C\A^3 \gamma^3 \int_{\{y=\epsilon\}} e^{2 \A \phi }y^{2a+2}(y^{a-1}\W)^2 \\
	& \rightarrow 0\ \text{as $\epsilon \to 0$ using \eqref{ftc}}.
\end{align*} 
Likewise using \eqref{ftc}, the second boundary integral can be estimated as 

\begin{align}
&\left| 2 \A \gamma c_s \int_{\{y=\epsilon\}} (y-y^a)y^{-2}W^2 \right| \leq C \alpha \gamma c_s \int_{\{y=0\}} e^{2\alpha \phi} (\py \W)^2\ \text{as $\epsilon \to 0$.}
\end{align}
Thus as $\epsilon \to 0$, we have \begin{align}\label{i6}
	&I_6\ge\frac{\A^3}{2} \gamma \int  (1-ay^{a-1})|x|^2W^2
+6 \A^3\gamma^3 \int (y-y^a)^2(1-ay^{a-1})W^2\\
& -2 \A \gamma c_s \int (y^{-2}+(a-2)y^{a-3})W^2 -C \alpha \gamma c_s \int_{\{y=0\}} e^{2\alpha \phi} (\py \W)^2.	\notag\end{align}
	
	Now for $I_7$,  using \eqref{Rellich2} in Lemma \ref{Rellich} with $k=a=0$, we find
	\begin{align*}
	 I_7=	2\int W_t\Delta W=- 2\int_{\{y=\epsilon\}}W_tW_y.
	\end{align*}
As before, the boundary integral is handled by writing $W$ in terms of $\W$, using $\W_t \to 0$ as $y \to 0$  and also by using \eqref{ftc} in the following way
\begin{align*}
	- 2\int_{\{y=\epsilon\}}W_tW_y&=-2 \int_{\{y=\epsilon\}} e^{2\alpha \phi}y^a \W_t \W_y -2\A \int_{\{y=\epsilon\}} e^{2\alpha \phi} (-\gamma y^a +\gamma y)y^a\W \W_t -a \int_{\{y=\epsilon\}}  e^{2\alpha \phi}y^{a-1}\W \W_t\\
	&\le C \left( \int_{\{y=\epsilon\}} e^{2\alpha \phi}(\py \W)^2 \int_{\{y=\epsilon\}} e^{2\alpha \phi} \W_t^2 \right)^{1/2} \\ &+ C\A \left(\int_{\{y=\epsilon\}} e^{2\alpha \phi} (y^{a-1}\W)^2 \int_{\{y=\epsilon\}} e^{2\alpha \phi} (y^{1+a}\W_t)^2 \right)^{1/2}\\
	& + C\left(\int_{\{y=\epsilon\}} e^{2\alpha \phi}(y^{a-1}\W)^2 \int_{\{y=\epsilon\}} e^{2\alpha \phi} \W_t^2  \right)^{1/2}\rightarrow 0\\& \text{as $\epsilon \to 0$ using \eqref{ftc} and also  that  $\W_t \to 0$ as $y \to 0$.}
\end{align*}
Hence in the limit $\epsilon \to 0$, we have \begin{align}\label{i7}
	I_7=0.
\end{align}
	Also by integrating by parts in the $t$ variable, we note	\begin{align}\label{i8}
		I_8=2\int(\A^2|x|^2/4 +\A^2\gamma^2 (y^a-y)^2+ c_s y^{-2}) W_t W= c_s\int  \partial_t ( y^{-2}W^2)=0.
	\end{align}
	
	Therefore using \eqref{bb1}-\eqref{i8}  and by rearranging the terms we have
	\begin{align}\label{kj1}
	&\int e^{2\A\phi}y^{-a}f^2
		\ge I_1 -2 \A \int |\n_x W|^2 + 4 \A \gamma \int (1-ay^{a-1})W_y^2\\
		&+(\A\gamma a (a-1)(a-2) - 2 \A \gamma c_s (a-2)+2\A \gamma a c_s)\int y^{a-3} W^2\notag\\
		&+\left(-(n+2)\frac{\A^3}{4}+n\frac{\A^3}{4}\right)\int |x|^2 W^2+(-n\A^3 \gamma^2+ n \A^3 \gamma^2)\int (y-y^a)^2W^2\notag\\
		&+(\A^3 \gamma/2 - \A^3 \gamma/2)\int (1-ay^{a-1})|x|^2W^2+(6\A^3 \gamma^3 -2 \A^3 \gamma^3)\int (y-y^a)^2(1-ay^{a-1})W^2\notag\\
		&+(-n\A c_s -2 \A \gamma c_s +n\A c_s -2 \A \gamma c_s) \int y^{-2}W^2 -16C\A\gamma \int_{\{y=0\}} e^{2\A\phi}(y^a\W_y)^2.\notag
	\end{align}
Now from the  Hardy's inequality as in \cite[Lemma 4.6]{RS}, we have
	\begin{align}\label{hrdy}
		\int W^2 y^{a-3} \le \frac{4}{(3-a-1)^2}\int y^{a-1}W_y^2 + \frac{2}{3-a-1}\int_{\{y=0\}} y^{a-2}W^2.
	\end{align}
	Note that $\A\gamma a (a-1)(a-2) - 2 \A \gamma c_s (a-2)+2\A \gamma a c_s= \A \gamma (1+2s)^2 a \le 0$, hence 
	\begin{align}
		(\A\gamma a (a-1)(a-2) - 2 \A \gamma c_s (a-2)+2\A \gamma a c_s)\int y^{a-3} W^2 \ge 4\A\gamma a \int y^{a-1}W_y^2 + 2\A\gamma a (2-a) \int_{\{y=0\}} y^{a-2}W^2.
	\end{align}
Now writing  $W$ in terms of $\W$ and then again by  using \eqref{ftc}  we get
\begin{align}
\left|	2\A \gamma a(2-a)\int_{\{y=0\}} y^{a-2}W^2 \right| \le 8 \A \gamma \int_{\{y=0\}}  e^{2\A \phi} y^{2a-2}\W^2 \le 8C\A \gamma \int_{\{y=0\}}  e^{2\A \phi} (\py \W)^2.
\end{align}
	Also, Since  $a \le 0,$ we have
	\begin{equation}\label{kjb}
	(y-y^a)^2(1-ay^{a-1}) \ge (y-y^a)^2.\end{equation}
	Hence  using \eqref{hrdy}-\eqref{kjb} in \eqref{kj1} we obtain the following estimate
	\begin{align}
	&\int e^{2\A\phi}y^{-a}f^2 \ge	I_1 -2 \A \int |\n_x W|^2 + 4 \A \gamma \int W_y^2-\frac{\A^3}{2}\int |x|^2 W^2
		+4\A^3 \gamma^3\int (y-y^a)^2W^2\\
		&\;\;\;-4 \A \gamma c_s \int y^{-2}W^2	- 24C\A \gamma \int_{\{y=0\}}  e^{2\A \phi} (\py \W)^2.\notag
	\end{align}
Since $\text{supp}(W(\cdot, t)) \subset \overline{ \mathbb B_{1/2}}$ and $c_s \leq 0$,  we get  from the above inequality that the following holds
\begin{align}\label{bb1p}
	\int e^{2\A\phi}y^{-a}f^2 &\ge	I_1 -2 \A \int |\n_x W|^2-\frac{\A^3}{8}\int  W^2
	+4\A^3 \gamma^3\int (y-y^a)^2 W^2	- 24C\A \gamma \int_{\{y=0\}}  e^{2\A \phi} (\py \W)^2.
\end{align}
	Now we absorb the term $-2 \A \int |\n_x W|^2$ using $I_1$. Note that from an integration by parts  argument
	\begin{align}\label{kjh}
		&-\A \int (\Delta W + \A^2 (|x|^2/4 +\gamma^2 (y^a-y)^2)W + c_s Wy^{-2})W\\
		&=\A \int |\n W|^2 + \alpha  \int_{\{y=0\}}WW_y - \A^3\int (|x|^2/4 +\gamma^2 (y^a-y)^2)W^2 -\A c_s \int W^2 y^{-2}.\notag
	\end{align}
	Now since  $ \text{supp}(W(\cdot, t)) \in \overline{\mathbb B_{1/2}}$, we get the following lower bound  on the last two terms in \eqref{kjh} above
	\begin{align*}
		- \A^3\int (|x|^2/4 +\gamma^2 (y^a-y)^2)W^2 \ge - \frac{\A^3}{16}\int W^2 -\A^3 \gamma^2 \int (y^a-y)^2)W^2.
	\end{align*}
Also, by substituting $W$ in terms of $\W$ and by using \eqref{ftc} it is seen that $\int_{\{y=0\}}WW_y =0.$\\
	Now using Cauchy-Schwarz inequality we get
	\begin{align}\label{li1}
		&-2\A \gamma \int (\Delta W + \A^2 (|x|^2/4 +\gamma^2 (y^a-y)^2)W + c_s Wy^{-2})W\\
		&\le  \int (\Delta W + \A^2 (|x|^2/4 +\gamma^2 (y^a-y)^2)W + c_s Wy^{-2})^2 + \A^2 \gamma^2 \int W^2\notag\\
		&= I_1 +\A^2 \gamma^2 \int W^2.\notag	\end{align}
	Hence  from \eqref{kjh}-\eqref{li1} we get the lower bound for $I_1$ 
	\begin{align}\label{li2}
		&I_1=\int (\Delta W + \A^2 (|x|^2/4 +\gamma^2 (y^a-y)^2)W + c_s Wy^{-2})^2\\
		&\ge -2\A \gamma \int (\Delta W + \A^2 (|x|^2/4 +\gamma^2 (y^a-y)^2)W + c_s Wy^{-2})W -\A^2 \gamma^2 \int W^2\notag \\
		& \ge 2\A \gamma \int |\n W|^2- \frac{2 \gamma\A^3}{16}\int W^2 -2\A^3 \gamma^3 \int (y^a-y)^2)W^2 -\A^2 \gamma^2 \int W^2.\notag
	\end{align}
	Using \eqref{li2} in \eqref{bb1p} we thus obtain
	\begin{align}\label{ufinal}
	&\int e^{2\A\phi}y^{-a}f^2\ge	-2 \A \int |\n_x W|^2 -\frac{\A^3}{8}\int W^2
		+4\A^3 \gamma^3\int (y-y^a)^2W^2\\
		&+2\A \gamma \int |\n W|^2- \frac{2 \gamma\A^3}{16}\int W^2 -2\A^3 \gamma^3 \int (y^a-y)^2)W^2 -\A^2 \gamma^2 \int W^2\notag\\
		&= 2\A(\gamma -1) \int |\n W|^2 -(\A^3/8 + \A^3 \gamma/8 + \A^2 \gamma^2)\int W^2 +2\A^3 \gamma^3 \int (y^a-y)^2)W^2\notag\\
		&\;\;- 24C\A \gamma \int_{\{y=0\}}  e^{2\A \phi} (\py \W)^2.\notag
	\end{align}
Again since  $\text{supp}(W(\cdot, t)) \subset \overline{\mathbb B_{1/2}}$, we have $(y^a-y)^2 \ge 1/4$.  At this point, by letting $\gamma=2$, from \eqref{ufinal} we can now assert that  there exist $C=C(n,s)$ such that  for all  $\A \ge 2$ we have 
$$\A^3 \int e^{2\A\phi}y^{a}\W^2\le C\int e^{2\A\phi}y^{-a}f^2 + C\A\int_{\{y=0\}}  e^{2\A \phi} (\py \W)^2.$$
This completes the proof.
	\end{proof}	
\subsection{Propagation of smallness estimate}
Using the Carleman estimate in Lemma \ref{carlbdr}, we now establish a quantitative propagation of smallness  from the boundary to the bulk as in Lemma \ref{bulkbdrinter} below.  This is a parabolic generalization of  Proposition 5.10 in \cite{RS} where we again  adapt a beautiful idea  from the proof of \cite[Theorem 15]{Ve}. Over here, one should note  that although even   the derivation of the estimate \eqref{vc} in the proof  of  Lemma \ref{bulkdf} involves  similar ideas from \cite{Ve}.  We  however  provide  all the details only in the case of  Lemma \ref{bulkbdrinter} below  and not for the estimate \eqref{vc}  above   because the weights  in the Carleman estimate \eqref{carlbdrest} are of a different kind  compared to the radial singular weights  in \eqref{carleman} which  are instead akin to that in \cite{Ve}.

Similar to that in \cite{RS}, we first  define the following sets for $s \in [1/2,1)$ which are tailored to the geometry of Carleman weights employed in Lemma \ref{carlbdr}. For a given $r>0$, we let
\begin{align*}
	P_r^*(x_0):&=\left\{X \in \R^{n} \times {\{y \geq 0\}} :y\le ((1-s)(r-|x-x_0|^2/4))^{\frac{1}{2-2s}}\right\}\\
	P_r(x_0):&=	P_r^*(x_0) \cap \{y=0\}.
\end{align*}
\begin{lemma}\label{bulkbdrinter}
	Let $s\in [1/2,1)$ and $\W$ be a solution to 
	\begin{align*}
		\D(y^a \n \W) + y^a \W_t&=0 \hspace{2mm}\text{in} \hspace{2mm} \R^{n+1}_+ \times \R,
	\end{align*}
	such that $\W=0$ in $\mathbb{B}_1 \times (1/8,7/8) \cap \{y=0\}$. Also assume that $\nabla_x \W,  \W_t$ are continuous up to the thin set $\{y=0\}$.  Then there exists $\theta \in (0,1)$ and $R_1 <1/2$ such that 
	\begin{align}\label{bulkbdr}
	||y^{a/2}\W||_{L^2(\mathbb B_{R_1}\times (1/4,3/4))} \le C ||y^{a/2}\W||_{L^2(\mathbb B_{1}\times (0,1))}^{\theta}||\py \W||_{L^2(B_1 \times (0,1))}^{1-\theta} + 2||\py \W||_{L^2(B_1 \times (0,1))}.
	\end{align}
\end{lemma}
\begin{proof}We work with $(-T,T)$ instead of $(0,T)$. Also we will write $\phi_{-}(r)=\underset{X\in \partial P_r^*}{\text{inf}}\phi(X)$ and $\phi_+(r)=\underset{X\in \partial P_r^*}{\text{sup}}\phi(X).$  We note that on $\partial P_r^*$, we have that \[ \frac{2y^{a+1}}{a+1} + \frac{|x|^2}{4} =r.\] Therefore from the definition of $\phi$ in Lemma \ref{carlbdr}, it follows that \begin{equation}\label{inf}  \text{inf}_{\partial P_r^*} \phi= \text{inf}_{ P_r^*} \phi = -r.\end{equation} Moreover, it is also seen that $\phi_+(r)$ is a decreasing function of $r$ for all $r \leq C(a)$.

Consider $W= \eta(t) \psi(X) \W$
where as before,  $\eta$ is a defined as
\[\eta(t)=\begin{cases}
	0 &\textrm{if}\;\;-T \le t \le T_1\\
	\operatorname{exp}\left(-\frac{T^3(T_2+t)^4}{(T_1+t)^3(T_1-T_2)^4}\right) &\textrm{if}\;\;-T_1 \le t \le T_1 \\
	1 &\textrm{if}\;\;-T_2 \le t \le 0
\end{cases}\]
$\eta(t)=\eta(-t)$ for $t>0,$ where $T_1=3T/4$ and $T_2=T/2,$
and $\psi$ is a smooth cut-off function symmetric in $y$  such that $\psi =1$ in $P_{r_0}^*$, $\psi =0$ in $\R^{n+1}_+ \setminus P_{2r_0}^*$. Note that one can ensure that  $|\partial_y \psi(x,y)| \le C(r_0)y$. We also note that $W$ solves 
\begin{equation}
\begin{cases}
	\D(y^a \n W) + y^a W_t&=\eta \W \D(y^a \n \psi) + 2\eta \langle \n \W, \n \psi\rangle  y^a +\eta_t \psi \W y^a   \hspace{2mm}\text{in} \hspace{2mm} \R^{n+1}_+ \times \R\\
	W=0\hspace{2mm} \text{on}\hspace{2mm} \{y=0\}.
	\end{cases}
\end{equation}
Also we observe  that $\text{supp}(W) \subset  P_{2r_0}^* \times (-T,T)$, where $r_0$ is choosen such that $ P_{2r_0}^* \subset \mathbb B_{1/2}.$
Now by applying the Carleman estimate \eqref{carlbdrest} in  Lemma \ref{carlbdr} to $W$, we obtain
\begin{align}\label{carlapp}
	\A^3 \int_{P_{2r_0}^* \times (-T,T)}e^{2\A \phi}W^2y^a &\le C \int_{P^*_{2r_0}\setminus P_{r_0}^*\times (-T,T)}e^{2\A \phi}g^2 y^{-a} + C \int_{P^*_{2r_0} \times (-T,T)}e^{2\A \phi}\eta_t^2 \W^2 \psi^2 y^{a}\\
	&+ C \A\int_{P_{2r_0}\times (-T,T)} e^{2\alpha \phi}(\py \W)^2 \eta^2 \psi^2=I_1+I_2+I_3.\notag
\end{align}
where \begin{equation}\label{gexp}g=\eta \W \D(y^a \n \psi) + 2\eta \langle \n \W, \n \psi\rangle  y^a.\end{equation} We observe that $g$ is supported on $P_{2r_0}^* \setminus P_{r_0}^*$.  We also note that $\eta_t$ is supported in $(-T_1,-T_2) \cup (T_2,T_1).$ We will only estimate $I_2$ in the region $P_{2r_0}^*\times (-T_1,-T_2)$ since other estimate is similar.
We write $P_{2r_0}^*\times (-T_1,T_2)=D \cup (P_{2r_0}^* \times (-T_1,T_2) \setminus D),$ with $D$ being the set where
$$C\frac{\eta_t^2}{\eta^2} > \frac{\A^3}{2}.$$
Now by a direct computation, it is easily seen that the  above inequality is equivalent to insisting
\begin{equation}\label{tyui1}2C \frac{16 T^6 }{(T_1+t)^8} > \A^3.\end{equation}
Hence for large $\A$ (depends also on $T$), $t\in (-T_1,-T_2)$, we get that \eqref{tyui1} implies $$\frac{T_1+t}{T} <\frac{1}{12}.$$
Now since $T_1-T_2=T/4$, we consequently obtain from the above inequality $$|T_2+t| > \frac{T}{6}.$$
Moreover, as  $\eta$ decays as exponential near $T_1$ we get that there exists a universal constant $C_1$ such that  \begin{equation}\label{al1}\frac{\eta_t^2}{\eta^2}\eta \le C_1.\end{equation} Also, Note that in $D,$ we have for large $\A$
\begin{align}\label{lD}
	&\operatorname{ln} CC_1 - 2\A  \phi_+(r_0) + \operatorname{ln} \eta\\
	&=	\operatorname{ln} CC_1  -2\A \phi_+(r_0) -\frac{T^3(T_2+t)^4}{(T_1+t)^3(T_1-T_2)^4}\notag\\
	& \le \operatorname{ln} CC_1  -2\A \phi_+(r_0) - C(T)\A^{9/8} \le 0.\notag
\end{align}
 We would like to emphasise the last inequality is possible because   the exponent of  $\A$ in $C(T) \alpha^{9/8}$  is strictly greater than $1$ and this is precisely where  the presence of the power  $\A^3$ in front of the integral $\int e^{2\A \phi}\W^2y^a$ in \eqref{carlbdrest} plays a crucial  role.  Thus for all  large $\A'$s, we have \begin{equation}\label{lop} CC_1 \eta \le e^{2\A \phi_+(r_0)}.\end{equation}
Using this, we estimate $I_2$  as follows.
\begin{align}\label{i21}
	C \int_{P_{2r_0}^*\times (-T_1,-T_2)}e^{2\A \phi}\eta_t^2 \W^2 \psi^2 y^a &=C \int_{D}e^{2\A \phi}\eta_t^2 \W^2 \psi^2 y^a+C \int_{P_{2r_0}^* \times (-T_1,-T_2)\setminus D}e^{2\A \phi}\eta_t^2 \W^2 \psi^2 y^a\\
	&\le C \int_{D}\left(\frac{\eta_t^2}{\eta^2}\eta\right)\eta \W^2  y^a+C \int_{P_{2r_0}^* \times (-T_1,-T_2)\setminus D}e^{2\A \phi}\frac{\eta_t^2}{\eta^2}\eta^2 \W^2 \psi^2 y^a\notag\\
	& \le e^{2 \A \phi_+(2r_0)} \int_{P_{2r_0}^*\times (-T,T)} \W^2 y^a + \frac{\A^3}{2} \int_{P_{2r_0}^*\times (-T,T)}e^{2\A \phi}W^2 y^a\notag\\  &\text{(using \eqref{al1} and \eqref{lop}).}\notag
\end{align}
Now since $\psi$ is smooth and symmetric in $y$ across $\{y=0\}$, we have  $|\psi_y| \le C(r_0)y$  moreover $\nabla^2 \psi$ is bounded. 
Now since  $\W=0$ on $\{y=0\}$, given the expression of $g$ as in \eqref{gexp}, we get from a Caccioppoli type energy estimate that the following holds
\begin{align}\label{i1energy}
	\int_{(P_{2r_0}^* \setminus P_{r_0}^*)\times (-T,T)}  e^{2\alpha \phi} g^2 y^a \le C  e^{2\alpha \phi_+(r_0)} \int_{\mathbb B_1 \times (-T,T)}\W^2y^a.
\end{align}
Here we used that $\phi_+$ is a decreasing function of $r$.
 Using \eqref{i21} and \eqref{i1energy} in \eqref{carlapp}  we thus obtain 
\begin{align}\label{try}
	\frac{\A^3}{2} \int_{P_{2r_0}^*\times (-T,T)}e^{2\A \phi}W^2 y^a \le Ce^{2 \A \phi_+(r_0)} \int_{\mathbb B_1 \times (-T,T)} \W^2 y^a +  C \A\int_{P_{2r_0} \times (-T,T)}(\py \W)^2.
\end{align}
We now minorize the integral on the left hand side  in \eqref{try}  over  the set $P_{r_0/2}^*\times (-T_2,T_2)$. Furthermore, we note that from \eqref{inf} it follows that on the set $P_{r_0/2}^*$, we have
\begin{equation}\label{inf1}
e^{2\alpha \phi} \geq e^{2\alpha \phi_{-}(r_0/2)}.
\end{equation}

Thus using \eqref{inf1} in \eqref{try}  we deduce that the following holds for all $\A$ large enough
\begin{equation}\label{kjhg}\int_{P_{r_0/2}^*\times (-T_2,T_2)}\W^2 y^a \le  e^{2 \A (\phi_+(r_0)-\phi_{-}(r_0/2))}\int_{B_1^* \times (-T,T)}\W^2y^a + e^{-2\A \phi_{-}(r_0/2)}\int_{B_1 \times (-T,T)}(\py \W)^2.\end{equation}
We now observe that  $\phi_{+}(r_0)-\phi_{-}(r_0/2) < 0$. This is seen as follows.
\begin{align}\label{yu}
	\phi_+(r_0) \le -r_0 +y^2 \le -r_0 +r_0^2/4 \le -\frac{3r_0}{4} < -r_0/2= \phi_{-}(r_0/2)\le 0.
\end{align}
In \eqref{yu} we used that since
$s\ge 1/2$, therefore we  have $y^2 \le ((1-s)(r_0-|x|^2/4))^{1/1-s} \le r_0^2/4.$ We now split the rest of the argument into two cases.

\emph{Case 1:} When $\int_{\mathbb B_1 \times (-T,T)}\W^2y^a \leq 2 \int_{B_1 \times (-T,T)}(\py \W)^2$. In this case, the desired estimate \eqref{bulkbdr} is seen to hold.

\medskip

\emph{Case 2:} When $\int_{\mathbb B_1 \times (-T,T)}\W^2y^a > 2 \int_{B_1 \times (-T,T)}(y^a \W_y)^2$. In this case, we let
$$ \A = \frac{1}{-2\phi_+(r_0)}\operatorname{ln}\left(\int_{\mathbb B_1 \times (-T,T)}\W^2y^a\middle/\int_{B_1 \times (-T,T)}(\py \W)^2\right) + C_0$$  in \eqref{kjhg} above, where $C_0$ is  a universal constant, we again conclude that \eqref{bulkbdr} holds in view of the fact that $P_{r_0/2}^*$ contains $\mathbb B_{R_1}$ for some $R_1$ depending only $r_0, n, a$.
This completes the proof.
\end{proof}
From Lemma \ref{bulkbdrinter}, the following propagation of smallness estimate follows which relates smallness in the bulk to the $C^{2}$ norm in the boundary. The proof of such a  lemma  is inspired by ideas  in \cite{Lin}.
\begin{lemma}\label{bbinterf}
	Let $U$ solve \eqref{extpr}. Then there exist universal $C=C(n,s)$ such that
	\begin{align*}
		||y^{a/2}U||_{L^2(\mathbb B_{R_1} \times (1/4,3/4))} &\le C||u||_{L^2(\R^{n+1})}^{1-\theta} \left(||Vu||^{\theta}_{L^2(B_1 \times (0,1))}+ ||u||^{\theta}_{C^2(B_2 \times (0,1))}\right)\\
		&\;\;\; +C(||Vu||_{L^2(B_1 \times (0,1))}+ ||u||_{C^2(B_2 \times (0,1))}),
	\end{align*}
where $R_1$ and $\theta$ are as in Lemma \ref{bulkbdrinter}.
\end{lemma}
\begin{proof}
Let $\eta \in C_c^{\infty}(B_2 \times (0,1))$ be  a cutoff function such that $\eta =1$ in $B_1 \times (1/8,7/8)$. First note that if $u \in C^2$ then $\eta u \in \text{Dom}(H^s)$ as
\begin{align*}
	&\int (1+| 4 \pi^2 |\xi|^2+ 2 \pi i \sigma| )^{2s}|\hat{\eta u}(\xi, \sigma)|^2\\
	&\le C \int (1+| \xi|^4+  \sigma^2 )|\hat{\eta u}(\xi, \sigma)|^2\\
	&\le C\left( ||u||_{C^2(B_2 \times (0,1))} \right)^2.
\end{align*}
Now let $W$ be a solution of 
\begin{align*}
	\D(y^a \n W) + y^a W_t&=0 \hspace{2mm}\text{in} \hspace{2mm} \R^{n+1}_+ \times \R\\
	W&=\eta u\hspace{2mm} \text{on}\hspace{2mm} \{y=0\}.
\end{align*}
Since $\eta u \in \text{Dom}(H^s)$, we have
\begin{equation}\label{no1}\left(\int_{\{y=0\}}(\py W)^2\right)^{1/2} = ||H^s (\eta u)||_{L^2} \le ||u||_{C^2}.\end{equation}
Now consider $\tilde{W}= U-W$ then we have $\tilde{W}=0$ on $B_1.$ We note that by an odd reflection of $\tilde W$ across $\{y=0\}$ and by arguments as in Section 5 in \cite{BG}, one can show that $\tilde W$ satisfies regularity assumptions in Lemma \ref{bulkbdrinter}.  Therefore by applying the estimate \eqref{bulkbdr} in Lemma \ref{bulkbdrinter} to $\tilde W$ we get
\begin{align*}
	||y^{a/2}\tilde{W}||_{L^2(\mathbb B_{R_1} \times (1/4,3/4))} \le C||y^{a/2}\tilde{W}||^{1-\theta}_{L^2(\mathbb B_1 \times (0,1))}||\py \tilde W||_{L^2(B_1\times (0,1))}^{\theta} +2||\py \tilde W||_{L^2(B_1\times (0,1))}
\end{align*} 
We now deduce from the above inequality
\begin{align}
	&||y^{a/2}U||_{L^2(\mathbb B_{R_1} \times (1/4,3/4))} \le ||y^{a/2} \tilde W||_{L^2(\mathbb B_{R_1} \times (1/4,3/4))} + ||y^{a/2}W||_{L^2(\mathbb B_{R_1}\times (1/4,3/4))}\\
	& \le C||y^{a/2}\tilde{W}||^{1-\theta}_{L^2(\mathbb B_1 \times (0,1))}||\py \tilde W||_{L^2(B_1\times (0,1))}^{\theta} +2||\py \tilde W||_{L^2(B_1\times (0,1))}\notag\\
	& \;\;\; + ||y^{a/2}W||_{L^2(\mathbb B_{R_1}\times (1/4,3/4))}.\notag
	\end{align} 
Also we have ( see for instance Lemma 4.5 in \cite{BG}) that the following holds \begin{equation}\label{no2}||y^{a/2}W||_{L^2(\R^{n} \times (0,1) \times (0,\infty))} \le C||\eta u||_{L^2(\R^{n+1})} \le C||u||_{L^2(B_2\times (0,1))}\end{equation} 
and  \begin{equation}\label{no4}||y^{a/2}U||_{L^2(\mathbb B_1 \times (0,1))}\le C ||u||_{L^2(\R^{n+1})}.\end{equation} Thus  from \eqref{no1}-\eqref{no4} we have 
\begin{align*}
	||y^{a/2}U||_{L^2(\mathbb B_{R_1} \times (1/4,3/4))} &\le C||u||_{L^2(\R^{n+1})} ^{1-\theta}\left(||\py U||^{\theta}_{L^2(B_1 \times (0,1))}+ ||u||^{\theta}_{C^2(B_2 \times (0,1))}\right)\\
	&\;\;\; + C(||\py U||_{L^2(B_1 \times (0,1))}+ ||u||_{C^2(B_2 \times (0,1))}),
\end{align*}
from which the desired estimate in the Lemma  follows using $\py U= Vu$.
\end{proof}

By a translation in time, the following corollary follows  from Lemma \ref{bbinterf} which will be more convenient to use in the  proof of Theorem \ref{mainthrm}.
\begin{cor}\label{remark}
		Let $U$ be as in Lemma \ref{bbinterf}. Then the following inequality holds		\begin{align}\label{shft}
			&||y^{a/2}U||_{L^2(\mathbb B_{R_1} \times (-R_1^2,R_1^2))} \le C||u||_{L^2(\R^{n+1})}^{1-\theta} \left(||Vu||^{\theta}_{L^2(B_1 \times (-1,1))}+ ||u||^{\theta}_{C^2(B_2 \times (-1,1))}\right)\\
			&\;\;\; +C(||Vu||_{L^2(B_1 \times (-1,1))}+ ||u||_{C^2(B_2 \times (-1,1))}),\notag
		\end{align}
		where $R_1$ and $\theta$ are as in Lemma \ref{bulkbdrinter}.
\end{cor}

With Lemma \ref{bulkdf} and Corollary \ref{remark} in hand, we now proceed with the proof of Theorem \ref{mainthrm}.

\begin{proof}[Proof of Theorem \ref{mainthrm}]
	Let $U$ be the solution to the  corresponding extension problem \eqref{extpr}. We choose  $\rho$ such that $\rho < \text{min}\{R_0 /8,R_1^2/8\}$  where $R_0$ and $R_1$ are as in   Lemma \ref{carlmanbulk} and Lemma \ref{bbinterf} respectively. Since $\rho < R_0/8$, from Lemma \ref{bulkdf}, we have
	\begin{equation}\label{cr1}\int_{\mathbb B_\rho\times (0,1)}U^2 y^a >C \rho^{A}, \end{equation} where $A$ is as in Lemma \ref{bulkdf}.
	The estimate \eqref{cr1}  implies  that there exist $t_0 \in (\rho^2,1-\rho^2)$ such that 
	\begin{align}\label{l1}
		 \int_{\mathbb B_{\rho}\times (t_0-\rho^2,t_0+\rho^2)}U^2y^a \geq C\rho^{A+2}.	\end{align}
	We now rescale  $U$ to $\tilde{U}$ as   follows
	$$\tilde{U}(x,y,t)=U\left(\frac{\rho}{R_1} x,\frac{\rho}{R_1} y,\left(\frac{\rho}{R_1}\right)^2 t+t_0\right),$$
  Note that $\tilde{U}$ solves
	\begin{align*}
		\D(y^a \n \tilde{U}) + y^a \tilde{U}_t&=0 \hspace{2mm}\text{in} \hspace{2mm} \R^{n+1}_+\times \R\\
		\tilde{U}&=\tilde{u}(x,t) \hspace{2mm} \text{on}\hspace{2mm} \{y=0\},
	\end{align*}
where $\tilde{u}(x,y,t)=u\left(\frac{\rho}{R_1} x,(\frac{\rho}{R_1})^2t+t_0)\right).$ Since $\py U =Vu,$
	we  have 
	\begin{align}\label{neu}
	\py \tilde U=\left(\frac{\rho}{R_1}\right)^{1-a}(Vu)\left(\frac{\rho}{R_1} x,\left(\frac{\rho}{R_1}\right)^2 t+t_0\right).
	\end{align}
	Then from Corollary  \ref{remark}, we have
		\begin{align}\label{eqrm}
		||y^{a/2}\tilde{U}||_{L^2(\mathbb B_{R_1} \times (-R_1^2,R_1^2))} &\le C||\tilde{u}||_{L^2(\R^{n+1})}^{1-\theta} \left(||\py \tilde{U}||^{\theta}_{L^2(B_1 \times (-1,1))}+ ||\tilde{u}||^{\theta}_{C^2(B_2 \times (-1,1))}\right)\\
		&\;\;\; +C(||\py \tilde{U}||_{L^2(B_1 \times (-1,1))}+ ||\tilde{u}||_{C^2(B_2 \times (-1,1))}).\notag
	\end{align}
  By change of variables and the assumption $||u||_{L^2(\R^{n+1})} \le 1$, we get
  \begin{align}\label{cov1} ||\tilde{u}||^{1-\theta}_{L^2(\R^{n+1})}=(\rho/R_1)^{-(1-\theta)\frac{n+2}{2}}||u||^{1-\theta}_{L^2(\R^{n+1})} \le (\rho/R_1)^{-(1-\theta)\frac{n+2}{2}}.
  \end{align}
Now using \eqref{neu} and change of variables we find
\begin{align}\label{cov2}
	&||\py \tilde U||_{L^2(B_1 \times (-1,1))} \leq ||V||_{L^{\infty}}\left(\frac{\rho}{R_1}\right)^{(1-a)}||\tilde{u}||_{L^2(B_1 \times (-1,1))} \\
	&=||V||_{L^{\infty}}\left(\frac{\rho}{R_1}\right)^{(1-a)}\left(\frac{\rho}{R_1}\right)^{-\frac{n+2}{2}}||u||_{L^2(B_{\rho/R_1}\times (-(\rho/R_1)^2+t_0,(\rho/R_1)^2+t_0))} \notag \\
	&=||V||_{L^{\infty}}\left(\frac{\rho}{R_1}\right)^{(1-a)}\left(\frac{\rho}{R_1}\right)^{-\frac{n+2}{2}}C_n \left(\frac{\rho}{R_1}\right)^{\frac{n+2}{2}} ||u||_{L^{\infty}(B_{\rho/R_1}\times (-(\rho/R_1)^2+t_0,(\rho/R_1)^2+t_0))}.\notag
\end{align}
Also, since $\rho < R_1^2/8$, we have $2\rho/R_1 <1$  which in particular implies
 \begin{align}\label{cov3}
 	||\tilde{u}||_{C^2(B_2 \times (-1,1))}\leq ||u||_{C^2(B_{2\rho/R_1}\times (-(\rho/R_1)^2+t_0,(\rho/R_1)^2+t_0))}.
 \end{align}
Now by  using \eqref{cov1}, \eqref{cov2}, \eqref{cov3}  in \eqref{eqrm} and also that $\left(\frac{\rho}{R_1}\right)^{(1-a)} \le 1$, we obtain
	\begin{align*}
		&(\rho/R_1)^{-(n+3+a)/2}||y^{a/2}U||_{L^2(\mathbb B_{\rho}\times (-\rho^2+t_0,\rho^2+t_0))}\\ & \le 2C (C_n^{\theta}+||V||^{\theta}_{L^{\infty}})(\rho/R_1)^{-\frac{(n+2)}{2}(1-\theta)} ||u||^{\theta}_{C^2((B_{2\rho/R_1}\times (-(\rho/R_1)^2+t_0,(\rho/R_1)^2+t_0))}\\
		&\;\;\;+2C(C_n+||V||_{L^{\infty}})||u||_{C^2((B_{2\rho/R_1}\times (-(\rho/R_1)^2+t_0,(\rho/R_1)^2+t_0))}.
	\end{align*}
Then by  multiplying the above inequality on both sides with  $(\rho/R_1)^{\frac{(n+2)}{2}(1-\theta)}$  and using $(\rho/R_1)^2+t_0<1+t_0<2$, $-(\rho/R_1)^2+t_0>-1+t_0>-1$, we get
	\begin{align}\label{pf}
	&(\rho/R_1)^{-((n+2)\theta +1 +a)/2}||y^{a/2}U||_{L^2(\mathbb B_{\rho}\times (-\rho^2+t_0,\rho^2+t_0))}\\ & \le 2C (C_n^{\theta}+||V||^{\theta}_{L^{\infty}}) ||u||^{\theta}_{C^2((B_{2\rho/R_1}\times (-1,2)}\notag\\
	&+2C(\rho/R_1)^{\frac{(n+2)}{2}(1-\theta)}(C_n+||V||_{L^{\infty}})||u||_{C^2((B_{2\rho/R_1}\times (-1,2)}.\notag
\end{align} 
Since $\theta <1$ and $\rho/R_1 <1$, we have $(\rho/R_1)^{\frac{(n+2)}{2}(1-\theta)}<1.$ Also there exists universal constant $C=C(\theta,s)$ such that we have $$(C_n^{\theta}+||V||^{\theta}_{L^{\infty}}) \le C 2^{||V||_{L^{\infty}}^{1/2s}}$$ and also $$(C_n+||V||_{L^{\infty}}) \le C 2^{||V||_{L^{\infty}}^{1/2s}}.$$ Hence \eqref{pf} can be written as
	\begin{align*} & (\rho/R_1)^{-((n+2)\theta +1 +a)/2}||y^{a/2}U||_{L^2(\mathbb B_{\rho}\times (-\rho^2+t_0,\rho^2+t_0))}\\ & \le C(\theta,s) 2^{||V||_{C^1}^{1/2s}}( ||u||^{\theta}_{C^2(B_{2\rho/R_1}\times (-1,2)}+||u||_{C^2(B_{2\rho/ R_1}\times (-1,2))}).\end{align*}
	Using \eqref{l1} in the above inequality we get
\begin{equation}\label{poi}C(\rho/R_1)^{-((n+2)\theta +1 +a)/2} \rho^{A/2+1} \le	C(\theta,s) 2^{||V||_{C^1}^{1/2s}}( ||u||^{\theta}_{C^2(B_{2\rho/R_1}\times (-1,2)}+||u||_{C^2(B_{2\rho/R_1}\times (-1,2))}).\end{equation}
Now multiplying on both sides of the inequality in \eqref{poi} with $(2/R_1)^{A/2+1}$ we obtain 
\begin{equation}\label{poi1}C(\rho/R_1)^{-((n+2)\theta +1 +a)/2} (2\rho/ R_1)^{A/2+1} \le	C(\theta,s)(2/R_1)^{A/2+1} 2^{||V||_{C^1}^{1/2s}}( ||u||^{\theta}_{C^2(B_{2\rho/R_1}\times (-1,2)}+||u||_{C^2(B_{2\rho/R_1}\times (-1,2))}),\end{equation}
where $A=C ||V||_{C^1}^{1/2s} +C_1$ is as in Lemma \ref{bulkdf}. Without loss of generality, we may assume that $C>2,$ then $\frac{A}{2} \ge ||V||_{C^1}^{1/2s}$ and thus the inequality in \eqref{poi1} above can be rewritten as 
\begin{align}\label{poi4}
&C(\rho/R_1)^{-((n+2)\theta +1 +a)/2} (2\rho/ R_1)^{A/2+1}\\
&\le	C(\theta,s)(2/R_1)^{\frac{A}{2}+1} 2^{\frac{A}{2}}( ||u||^{\theta}_{C^2(B_{2\rho/R_1}\times (-1,2)}+||u||_{C^2(B_{2\rho/R_1}\times (-1,2))})\notag\\
&=	C(\theta,s)(4/R_1)^{\frac{A}{2}+1}(2/R_1)( ||u||^{\theta}_{C^2(B_{2\rho/R_1}\times (-1,2)}+||u||_{C^2(B_{2\rho/R_1}\times (-1,2))}).\notag
\end{align}
Since $8\rho/R_1^2<1$, we have $(8\rho/R_1^2)^{\frac{A}{2}} \le 1,$ which gives \begin{equation}\label{poi5}(4/R_1)^{\frac{A}{2}} \le (2\rho/R_1)^{-\frac{A}{2}}.\end{equation} 	Using \eqref{poi5} in \eqref{poi4} we deduce the following estimate
\begin{align}\label{poi7}
	&C(\rho/R_1)^{-((n+2)\theta +1 +a)/2} (2\rho/ R_1)^{A/2+1}\\
	&\le	C(\theta,s)(2\rho/R_1)^{-\frac{A}{2}}(2/R_1)( ||u||^{\theta}_{C^2(B_{2\rho/R_1}\times (-1,2)}+||u||_{C^2(B_{2\rho/R_1}\times (-1,2))}).\notag
\end{align}
\eqref{poi7} can be rewritten as
\begin{align}\label{poi8}
	&C (2\rho/ R_1)^{A_1}
	\le	C(\theta,s)(2/R_1)( ||u||^{\theta}_{C^2(B_{2\rho/R_1}\times (-1,2)}+||u||_{C^2(B_{2\rho/R_1}\times (-1,2))}).
\end{align}
	where $A_1=A+1 -\frac{(n+2)\theta +1 +a}{2}= C||V||_{C^1}^{1/2s}+C_1+\frac{1-a-(n+2)\theta}{2}$.  Now since we are interested in a vanishing order estimate from below for small $\rho$, without loss of generality, we may assume that
	\begin{equation}\label{asum}
	||u||_{C^2(B_{2\rho/R_1}\times (-1,2))}< 1
	\end{equation}
	since otherwise the desired estimate in Theorem \ref{mainthrm} holds trivially. 	
	In view of \eqref{asum}, since $\theta <1 $, we can thus assert that
	\begin{equation}\label{th1}
	||u||_{C^2(B_{2\rho/R_1}\times (-1,2))} \leq 	||u||^{\theta}_{C^2(B_{2\rho/R_1}\times (-1,2))}.\end{equation}	  Thus using \eqref{th1} in \eqref{poi8} we obtain
	\begin{align}\label{final}
		\tilde C (2\rho/ R_1)^{\tilde A}
		\le	||u||_{C^2(B_{2\rho/R_1}\times (-1,2))}.
	\end{align}
where $\tilde A=A_1/\theta$ and $\tilde C= \tilde C(R_1, \theta, s, n).$ The desired vanishing order  estimate in Theorem \ref{mainthrm} thus follows from \eqref{final} in a standard way by letting $\frac{2\rho}{R_1}$ as the new $\rho$.
\end{proof}

\begin{proof}[Proof of Theorem \ref{mainthrm2}] Before we proceed with the proof, we would like to alert the reader that throughout the proof, we will use the letter $C$  to denote all purpose universal constant which might vary from line to line.

	We first note that  from a careful investigation of the proof of  Theorem \ref{mainthrm},  we find that  there exists $t_0$ such that 
	\begin{align}\label{df1}
		||u||_{W^{2,2}( B_{\rho} \times (t_0-{\rho}^2,t_0+{\rho}^2))} \ge C_1\rho^{A_0},
	\end{align}
where $C_1$ and $A_0$ is as in the Theorem \ref{mainthrm}. Overhere, $W^{2,2}$ norm of $u$ refers to
 $$ 	||u||_{W^{2,2}}\overset{def}=||u||_{L^2} + ||\n_x u||_{L^2}+ ||\n^2_x u||_{L^2}+ ||u_t||_{L^2}.$$ 
 	In the rest of the proof, for notational convenience we will denote $B_{\rho} \times (t_0-{\rho}^2,t_0+{\rho}^2)$ by $ Q_{\rho}$ and $\mathbb B_{\rho} \times (t_0-{\rho}^2,t_0+{\rho}^2)$ by $\mathbb Q_{\rho}.$ Also $||V||_{C^{2}_{x,t}}$ will be denoted by $||V||_{C^2}$. Then from the rescaled version of the estimate \eqref{ret} in Lemma \ref{reg1} we have
\begin{align}\label{rrg}
	||y^{a/2} \nabla U_{ij}||_{L^2(\mathbb Q_r)}  \le r^{-3} C(1+||V||_{C^2}) ||y^{a/2} U||_{L^2(\mathbb Q_{2r} )}
\end{align}
and also
\begin{align}\label{rrg1}
r^2||y^{a/2}  U_{t}||_{L^2(\mathbb Q_r)}+r^3||y^{a/2} \nabla U_{t}||_{L^2(\mathbb Q_r)} + r^4	||y^{a/2} U_{tt}||_{L^2(\mathbb Q_r)} \le  C(1+||V||_{C^2}) ||y^{a/2} U||_{L^2(\mathbb Q_{2r} )}.
\end{align}
Let $\phi$ be a smooth function  supported in $\overline{\mathbb B_{2{\rho}}}\times (t_0 - (2{\rho})^2, t_0 + (2{\rho})^2) $ such that $\phi \equiv 1$ in  
$\overline{\mathbb B_{\rho}} \times  (t_0-{\rho}^2,t_0+{\rho}^2).$ 
 We now apply interpolation inequality \eqref{inte}  to $f=\phi U$ and  obtain for any $0<\eta_1 <1$ that the following estimate holds
\begin{align}\label{der1}
	&||\n_x f||_{L^2(\R^{n+1})}\\ &\le C\left( \eta_1^{s}(||y^{a/2}\n \n_x f||_{L^2(\R^{n+1}\times \R_+)} +||y^{a/2}\n_xf||_{L^2(\R^{n+1}\times \R_+)})
	 + \eta_1^{-1}||f||_{L^2(\R^{n+1})} \right)\notag\\
	& \le C\big(\eta_1^{s}(|\phi|||y^{a/2}\n \n_xU||_{L^2(\mathbb Q_{2{\rho}})} +(|\n\phi|+|\phi|)||y^{a/2}\n U||_{L^2(\mathbb Q_{2{\rho}})}\notag\\
	& +(|\n^2 \phi|+|\n\phi|)||y^{a/2} U||_{L^2(\mathbb Q_{2{\rho}})})+ \eta_1^{-1}||u||_{L^2( Q_{2{\rho}})} \big)\notag\\
	 & \le   \eta_1^{s}{\rho}^{-2}C(1+||V||_{C^2}) ||y^{a/2} U||_{L^2(\mathbb Q_{4{\rho}} )} +C \eta_1^{-1}||u||_{L^2( Q_{2{\rho}})},\notag
\end{align} 
where  in the last inequality, we used the rescaled versions of the regularity estimate in Lemma \ref{reg1}.  From \eqref{der1} it follows
\begin{align}\label{d1f}
	||\n_x u||_{L^2(Q_{\rho})} \le  C \eta_1^{s}{\rho}^{-2}(1+||V||_{C^2}) ||y^{a/2} U||_{L^2(\mathbb Q_{4{\rho}} )} +C \eta_1^{-1}||u||_{L^2( Q_{2{\rho}})}.
\end{align}
Similarly  by applying \eqref{inte} to $\nabla_x f$ and also by using  \eqref{rrg} we get  for any $0<\eta<1$ 
\begin{align}\label{der2}
	||\n_x^2 u||_{L^2( Q_{\rho})} \le C\eta^{s}{\rho}^{-3}(1+||V||_{C^2}) ||y^{a/2} U||_{L^2(\mathbb Q_{4{\rho}} )} +C\eta^{-1}||\n_x f||_{L^2(\R^{n+1})}.
\end{align}
Now using \eqref{der1} in \eqref{der2} we obtain
\begin{align}\label{der22}
		||\n_x^2 u||_{L^2( Q_\rho)} &\le \eta^{s}{\rho}^{-3}C(1+||V||_{C^2}) ||y^{a/2} U||_{L^2(\mathbb Q_{4{\rho}} )} \\
		&\;\;\;+\eta^{-1}\eta_1^{s}{\rho}^{-2}C(1+||V||_{C^2}) ||y^{a/2} U||_{L^2(\mathbb Q_{4{\rho}} )} +C(\eta\eta_1)^{-1}||u||_{L^2( Q_{2{\rho}})}.\notag
\end{align}
 We now take $\eta_1 =\eta^3$.  Then we find 
\begin{align}\label{ec1}
\eta^{-1}\eta_1^{s}=	\eta^{3s-1}. \end{align}
Also observe that $3s-1 \geq s$ as $s \geq 1/2$.   Eventually we take $\eta <<1$. In view of this, we thus obtain from \eqref{der22}
\begin{align}\label{d2f}
	||\n^2_x u||_{L^2(Q_{\rho})} &\le \eta^{s}{\rho}^{-3}C(1+||V||_{C^2}) ||y^{a/2} U||_{L^2(\mathbb Q_{4{\rho}} )}
	 +C\eta^{-4}||u||_{L^2( Q_{2{\rho}})}.
\end{align}
Similarly by applying \eqref{inte1} to $f= U \phi$ and also by using the estimate \eqref{rrg1}  we get
\begin{equation}\label{dkf}
||u_t||_{L^2(Q_{\rho})} \leq \eta^{s}{\rho}^{-4}C(1+||V||_{C^2}) ||y^{a/2} U||_{L^2(\mathbb Q_{4{\rho}} )} + C\eta^{-1}||u||_{L^2( Q_{2{\rho}})}.\end{equation} Therefore from \eqref{d1f}, \eqref{d2f} and \eqref{dkf} we have
\begin{align}\label{nop}
	||u||_{W^{2,2}(Q_{\rho})} \le C\eta^{-4}||u||_{L^2(Q_{4{\rho}})} + (C/{\rho}^4)\eta^s 2^{||V||_{C^2}^{1/2s}}.
\end{align}
Note that in \eqref{nop} we used that $||y^{a/2} U||_{L^2(\mathbb Q_{4{\rho}} )} \leq C$. If we now  take $\eta^s = (C_1/2C){\rho}^{A_0 +4 + ||V||_{C^2}^{1/2s}}$ and use \eqref{df1}, we get 
\begin{equation}\label{khg}C_1{\rho}^{A_0} \le C \eta^{-4}||u||_{L^2(Q_{4{\rho}})} +  \frac{C_1}{2}{\rho}^{A_0}(2{\rho})^{||V||_{C^2}^{1/2s}}.\end{equation}
In \eqref{khg} above, we can now absorb the term $\frac{C_1}{2}{\rho}^{A_0}(2{\rho})^{||V||_{C^2}^{1/2s}}$ in the left hand side and  then the desired estimate \eqref{main2} is seen to follow. 

\end{proof}

\end{document}